\documentclass[a4paper,11pt]{article}
\pdfoutput=1
\usepackage{setspace}
\usepackage{hyperref}
\usepackage{caption}
\usepackage{graphicx}
\graphicspath{ {images/}}
\usepackage{wrapfig}
\usepackage{subcaption}
\usepackage{array}
\usepackage{tikz}
\usetikzlibrary{matrix}
\usetikzlibrary{cd}

\usepackage{dsfont}
\usepackage{mathrsfs}
\usepackage{amssymb,amsthm,array}
\usepackage[reqno]{amsmath}
\usepackage{textcomp}
\usepackage{theoremref}
\usepackage{calligra}
\usepackage[nottoc]{tocbibind}
\usepackage{epigraph}

\usepackage{bibleref}

\author{Philip J. Carter}
\title{On the Resolutions of Non-Dicritical Foliations}
\date{}

\theoremstyle{plain}
\newtheorem{theorem}{Theorem}[section]
\newtheorem{proposition}[theorem]{Proposition}
\newtheorem{lemma}[theorem]{Lemma}
\newtheorem{corollary}[theorem]{Corollary}

\theoremstyle{definition}
\newtheorem{definition}[theorem]{Definition}
\newtheorem{example}[theorem]{Example}
\newtheorem{remark}[theorem]{Remark}

\newtheorem{notation}[theorem]{Notation}
\newtheorem*{convention*}{Convention}

\numberwithin{equation}{section}

\newcommand{\rrl}{\mathds{R}}
\newcommand{\ccx}{\mathds{C}}
\newcommand{\ocurl}{\mathscr{O}}

\newcommand{\vv}{\mathds{V}}
\newcommand{\aaf}{\mathds{A}}
\newcommand{\qq}{\mathds{Q}}
\newcommand{\zz}{\mathds{Z}}
\newcommand{\nn}{\mathbb{N}}
\newcommand{\mm}{\mathfrak{m}}

\newcommand{\Rea}{\mathfrak{Re}}
\newcommand{\Ima}{\mathfrak{Im}}

\newcommand{\Hom}{\operatorname{Hom}}
\newcommand{\Homscr}{\mathscr{H}\text{\kern -3pt {\calligra\large om}}\,}

\newcommand{\Spec}{\operatorname{Spec}}
\newcommand{\dimn}{\operatorname{dim}}

\newcommand{\supp}{\operatorname{supp}}
\newcommand{\redn}{\operatorname{red}}
\newcommand{\ord}{\operatorname{ord}}

\newcommand{\Singr}{\operatorname{Sing}}

\newcommand{\logn}{\operatorname{ln}}
\newcommand{\Spf}{\operatorname{Spf}}
\newcommand{\sat}{\operatorname{sat}}
\newcommand{\Tors}{\operatorname{Tor}}

\newcommand{\fcurl}{\mathcal{F}}
\newcommand{\gcurl}{\mathcal{G}}

\newcommand{\icurl}{\mathcal{I}}
\newcommand{\Sscr}{\mathscr{S}}
\newcommand{\Cscr}{\mathscr{C}}
\newcommand{\Xcal}{\mathcal{X}}
\newcommand{\Dcal}{\mathcal{D}}
\newcommand{\Tcal}{\mathcal{T}}

\newcommand{\Bfrak}{\mathfrak{B}}

\newcommand{\restn}[1]{|_{#1}}

\begin{document}
\maketitle

\section*{Abstract}
\begin{singlespace}
\sloppy We introduce the jet schemes of a holomorphic foliation, and thereby prove an alternate characterisation of simple singularities of codimension-$1$ foliations, independent of any normal form. This leads to an equivalent condition for the existence of a desingularisation in the non-dicritical case. We then prove that such a desingularisation always exists, at least on the level of germs.
\end{singlespace}

\section*{Introduction}
Foliations of complex manifolds arise from holomorphic differential forms on the manifold, or by duality from an integrable sheaf of holomorphic vector fields. As with algebraic varieties, we seek a classification of foliations up to birational equivalence.
Just as the classification problem for varieties depends on the existence of a desingularisation -- first proved by Hironaka in \cite{Hir64} -- so in the case of foliations we seek to show the existence of a resolution of the singularities -- a sequence of blow-ups in smooth centres such that the pullback foliation has the best possible singularities.

We principally consider codimension-$1$ foliations, in which case the desired endpoint of the resolution is to have so-called simple singularities. Existence of resolutions for codimension-$1$ foliations has so far been proved for ambient dimension $2$ by Seidenberg in \cite{Sei68} and for dimension $3$ by Cano in \cite{Can04}. In addition, existence of resolutions for dimension-$1$ folations in ambient dimension $3$ was proved by Panazzolo and McQuillan in \cite{MP13}. For a review of other results pertaining to resolutions, see the introduction of \cite{Car21}.

The key result of this paper is \thref{bigthm4}: Restricting to the case of non-dicritical codimension-$1$ foliations, (a condition which ensures there exist, and are only finitely many, separatrices -- leaves of the foliation whose analytic closures pass through the singular locus), resolutions exist in any ambient dimension, globally in the case where all separatrices are analytic, and in all cases on the level of germs.

Following the spirit of \cite{Mus02}, where jet spaces are applied to the study of singularities of schemes, in this paper we introduce the jet space of a foliation, which allows us firstly to define a stronger notion of tangency to a foliation, such that non-reduced schemes can be considered as tangent, and secondly to produce a geometrically more natural characterisation of simple singularities, namely that  in a neighbourhood of a simple singularity, the foliation has the same geometric structure (as determined by the jet spaces) as a normal crossings divisor (see \thref{bigthm2}).

In the codimension-$1$ case, we can then prove \thref{bigthm3}: A non-dicritical foliation admits a resolution if and only if it admits a total separatrix (the maximal strongly tangent (formal) subscheme supported on the union of the separatrices), which can itself be resolved to a normal crossings divisor. (In the case of non-algebraic formal schemes, we restrict to the level of germs to ensure resolvability.) We then prove that the total separatrix always exists, completing the key proof.

This paper is an updated version of \cite{Car22}, with the addition of the proof of \thref{bigthm4}, and can be considered a replacement for it. The bulk of the material is based on my PhD thesis \cite{Car21}. Readers seeking more background, fuller exposition, or omitted proofs are directed to the full version of the thesis.

\section{Conventions and Background}

\begin{convention*}
We denote by $\nn$ the set of natural numbers $\{1,2,3,\ldots \}$. We denote by $\nn_0$ the set $\nn\cup \{0\}$.

\noindent All rings are assumed to be commutative with unit.

\noindent All schemes are defined over the field of complex numbers $\ccx$.

\noindent Any manifold $X$ is a complex manifold, with holomorphic tangent sheaf $\Tcal_X$, holomorphic tangent bundle $TX$, and sheaf of holomorphic $1$-forms $\Omega^1_X$.

\noindent We denote by $\hat{\mm}_n$ the ideal $(x_1,\ldots,x_n)\subset\ccx[[x_1,\ldots,x_n]]$.
\end{convention*}

We define a \emph{variety} to be a reduced, irreducible, separated scheme of finite type over $\ccx$.

\begin{notation}
Let $I=(f_1,\ldots,f_r)\subset\ccx[x_1,\ldots,x_n]$ be an ideal. The affine scheme $X=\Spec(\ccx[x_1,\ldots,x_n]/I)$ is also denoted by $\vv(I)$ or $\vv(f_1,\ldots,f_r)$.
\end{notation}

\section{Formal Schemes}\label{Formal}
We now introduce formal schemes. In the literature, several different classes of objects go by this name; in this paper we take as our definition that given by Yasuda in \cite{Yas07}, defined in terms of prorings, which we will define presently. Having introduced these formal schemes, following Yasuda's paper, we give some results relating this category to that of formal schemes as introduced by Grothendieck in \cite{DG67}, which we call classical formal schemes.

\subsection{Prorings}
\begin{definition}
A \emph{proring} is a directed projective system of rings $(R_{\lambda})_{\lambda\in\Lambda}$, where $\Lambda$ is a poset, with morphisms $f_{\lambda\mu}:R_{\mu}\rightarrow R_{\lambda}$.
\end{definition}

\begin{definition}
Let $R=(R_{\lambda})_{\lambda\in\Lambda}, S=(S_{\mu})_{\mu\in M}$ be two prorings. A \emph{morphism} of prorings $R\rightarrow S$ is an element of
\[
\Hom(R,S)=\varprojlim_{\mu}\varinjlim_{\lambda}\Hom(R_{\lambda},S_{\mu}),
\]
where the latter sets are morphisms of rings.
\end{definition}

Such a morphism of prorings $f:R\rightarrow S$ is defined by a collection of representative ring homomorphisms $f^{\lambda}_{\mu}:R_{\lambda}\rightarrow S_{\mu}$. Then the diagrams
\[
\begin{tikzcd}
R \arrow[r,"f"] \arrow[d] & S\arrow[d] \\
R_{\lambda} \arrow[r,"f^{\lambda}_{\mu}"] & S_{\mu}
\end{tikzcd}
\]
commute.

\begin{example}
By indexing a projective system over a singleton set, we see that every ring is a proring.
\end{example}

\begin{proposition}
The category of prorings has all projective limits.
\end{proposition}

\begin{proof}[Proof: \nopunct]
This is true of the category of rings, so the result follows by \cite[Proposition 1.2]{Yas07}.
\end{proof}

\begin{remark}
By \cite[Theorem 1]{HS54}, there is an uncountable projective system of rings, each of which is countably infinite, such that the projective limit is the zero ring. Hence the limit of a projective system of rings in the category of rings may not be isomorphic to its limit in the category of prorings.
\end{remark}

\begin{definition}
A proring is said to be \emph{epi} if all the morphisms $f_{\lambda\mu}$ are epimorphisms.
\end{definition}

For a topological space $X$ we can define a sheaf of prorings. Such a sheaf can be presented as a projective system of sheaves of rings. Therefore by \cite[Proposition 1.2]{Yas07} the category of sheaves of prorings has projective limits.

\subsection{Formal Schemes via Prorings}
\begin{definition}
Let $X$ be a topological space. We say that $X$ is \emph{quasi-separated} if the intersection of any two quasi-compact open subsets is quasi-compact.

We say that $X$ is \emph{qsqc} if it is quasi-separated and quasi-compact.

A \emph{qsqc basis} is a basis of open subsets, all of which are qsqc.
\end{definition}

Any scheme has a qsqc basis. An affine scheme, or any Noetherian topological space, is qsqc.

\begin{definition}
An \emph{admissible system of schemes} is a directed inductive system $(X_{\lambda})_{\lambda\in\Lambda}$ of schemes such that every morphism $X_{\lambda}\rightarrow X_{\mu}$ is a bijective closed immersion.
\end{definition}

\begin{definition}
A proring $R=(R_{\lambda})_{\lambda\in\Lambda}$ is \emph{admissible} if every morphism $R_{\lambda}\rightarrow R_{\mu}$ is surjective, and induces an isomorphism $(R_{\lambda})_{\redn}\rightarrow (R_{\mu})_{\redn}$.

Equivalently, $R$ is admissible if $(\Spec(R_{\lambda}))_{\lambda\in\Lambda}$ is an admissible system of schemes.
\end{definition}

An admissible proring $R$ has an associated reduced ring $R_{\redn}=(R_{\lambda})_{\redn}$, for any $\lambda\in\Lambda$.

\begin{definition}
A \emph{locally admissibly proringed space} is a pair $(X,\ocurl_X)$, where $X$ is a topological space with a qsqc basis $\Bfrak$, $\ocurl_X$ is a sheaf of prorings such that $\ocurl_X(U)$ is admissible for all $U\in\Bfrak$, and for each $x\in X$ the stalk $\ocurl_{X_{\redn},x}$, where $\ocurl_{X_{\redn}}$ is the sheaf of rings defined on $\Bfrak$ by $U\mapsto (\ocurl_X(U))_{\redn}$, is a local ring.
\end{definition}

\begin{definition}
Let $R=(R_{\lambda})_{\lambda}$ be an admissible proring. We define the \emph{formal spectrum} of $R$ to be the locally admissibly proringed space $\Spf(R)$, which as a topological space is equal to $\Spec(R_{\redn})=\Spec(R_{\lambda})$, for all $\lambda$, and whose structure sheaf is
\[
\ocurl_{\Spf(R)}=\varprojlim \ocurl_{\Spec(R_{\lambda})},
\]
where the limit is taken in the category of sheaves of prorings.
\end{definition}

\begin{definition}
A \emph{formal scheme} is a locally admissibly proringed space locally isomorphic to the formal spectrum of an admissible proring.

A formal scheme which is isomorphic to a formal spectrum is called \emph{affine}.
\end{definition}

\begin{example}
Any ring $R$ can be viewed as a proring, in which case it will be admissible, with $\Spf(R)=\Spec(R)$. Thus any scheme can be viewed as a formal scheme.
\end{example}

\begin{remark}
To distinguish from strict formal schemes (i.e. formal schemes which are not schemes), schemes are sometimes called \emph{ordinary schemes}.
\end{remark}

\begin{proposition}\cite[Corollary 2.14]{Yas07}
The category of affine formal schemes is equivalent to the dual category of the category of admissible prorings.
\end{proposition}

An affine formal scheme $X$ corresponds to an admissible system of affine schemes $(X_{\lambda})_{\lambda\in\Lambda}$. As such, we can view $X$ as the limit $X=\varinjlim X_{\lambda}$. A general formal scheme which is qsqc can also be viewed as the limit of an admissible system of schemes (see \cite[Proposition 3.32]{Yas07}). We therefore sometimes write formal schemes in this limit notation.

\begin{definition}
Let $X=\varinjlim X_{\lambda}$ be an affine formal scheme. $X$ is \emph{countably indexed} if the indexing set $\Lambda$ can be taken to be countable (say $\Lambda=\nn$).

If $X$ is a general formal scheme, we say it is \emph{locally countably indexed} if it admits a cover by countably indexed affine formal schemes.
\end{definition}

\begin{remark}
In \cite{Yas07}, the term \emph{gentle} is used in place of locally countably indexed.
\end{remark}

\begin{definition}
An admissible proring $R=(R_{\lambda})_{\lambda}$ is called \emph{pro-Noetherian} if every $R_{\lambda}$ is Noetherian.

A formal scheme $X$ is called \emph{locally ind-Noetherian} if every $x\in X$ admits an affine neighbourhood $\Spf(R)\subset X$, where $R$ is pro-Noetherian.
\end{definition}

\begin{definition}
Let $X=\varinjlim X_{\lambda}$ and $Y=\varinjlim Y_{\mu}$ be formal schemes. We say that $X\subset Y$ if for all $\lambda\in\Lambda$ there exists $\mu\in M$ such that $X_{\lambda}\subset Y_{\mu}$.
\end{definition}

\begin{remark}
Note that this is a more general notion than that of a formal subscheme in \cite{Yas07}, as we do not take into account the topologies on the formal schemes.
\end{remark}

\begin{proposition}
Let $(X_{\alpha})_{\alpha\in A}$ be an inductive system of formal schemes, all of which have the same underlying topological space. Then the direct limit $\varinjlim X_{\alpha}$ exists as a formal scheme.
\end{proposition}

\begin{proof}[Proof: \nopunct]
See \cite[Proposition 5.15]{Car21}.
\end{proof}

\subsection{Other Results}
\begin{definition}
A \emph{locally topologically ringed space} is a pair $(X,\ocurl_X)$, where $X$ is a topological space, and $\ocurl_X$ is a sheaf of topological rings whose stalks $\ocurl_{X,x}$ are local rings.
\end{definition}

\begin{definition}
Let $R$ be topological ring which is linearly topologised, complete and separated.

$R$ is \emph{admissible} if there exists an open ideal $I\subset R$ such that every neighbourhood of $0$ contains $I^n$ for some $n\in\nn$. Such an ideal is called an \emph{ideal of definition}.
\end{definition}

\begin{remark}
The collection of all ideals of definition forms a fundamental system of neighbourhoods of $0$.
\end{remark}

\begin{remark}
Any ring with the discrete topology is admissible, with the zero ideal being an ideal of definition.
\end{remark}

\begin{definition}
Let $R$ be an admissible ring, and let $(I_{\lambda})_{\lambda\in\Lambda}$ be the collection of all ideals of definition. We define the \emph{formal spectrum} of $R$ to be the locally topologically ringed space $\Spf(R)$, which as a topological space comprises the open prime ideals of $R$ with the subspace topology from $\Spec(R)$, or equivalently (by \cite[Lemma 5.21]{Car21}) is equal to $\Spec(R/I_{\lambda})$, and has the structure sheaf
\[
\ocurl_{\Spf(R)}=\varprojlim\ocurl_{\Spec(R/I_{\lambda})}.
\]
\end{definition}

\begin{definition}
A \emph{classical formal scheme} is a locally topologically ringed space $X$ locally isomorphic to the formal spectrum of an admissible ring.

A classical formal scheme which is isomorphic to a formal spectrum is called \emph{affine}.
\end{definition}

\begin{remark}
Any scheme is a classical formal scheme.
\end{remark}

\begin{proposition}\thlabel{fscor}
Any classical formal scheme $X$ corresponds to a unique formal scheme.

Any locally countably indexed, locally ind-Noetherian formal scheme $X$ corresponds to a unique classical formal scheme.
\end{proposition}

For the details of the proof, see \cite[Proposition 5.26, Proposition 5.27]{Car21}, the elements of which are found in \cite{Yas07}.

\begin{example}
Let $X$ be a smooth variety, and $Z\subset X$ be a closed reduced subscheme given by the ideal sheaf $\icurl$. For $n\in\nn$, let $Z^n$ be the scheme given by the ideal sheaf $\icurl^n$. The formal scheme $\hat{X}_Z=\varinjlim Z^n$ is called the \emph{formal completion} of $X$ along $Z$.

The formal completion is a classical formal scheme. In the affine case, with $X=\Spec(R)$ and $Z=\Spec(R/I)$, then $\hat{X}_Z=\Spf(\varprojlim R/I^n)$.
\end{example}

\begin{definition}
Let $X$ be a classical formal scheme with affine cover $\bigcup \Spf (R_{\alpha})$. A \emph{closed formal subscheme} $Z\subset X$ is the classical formal scheme with affine cover $\bigcup \Spf(R_{\alpha}/I_{\alpha})$, where the $I_{\alpha}$ are closed ideals forming a coherent sheaf.
\end{definition}

\begin{example}
For formal power series $g_1,\ldots, g_r\in \ccx[[x_1,\ldots,x_n]]$ we can define the formal scheme 
\[
\vv(g_1,\ldots,g_r)=\Spf(\ccx[[x_1,\ldots,x_n]]/(g_1,\ldots,g_r)),
\]
which is a closed formal subscheme of $\Spf(\ccx[[x_1,\ldots,x_n]])$. Furthermore, any closed formal subscheme of $\Spf(\ccx[[x_1,\ldots,x_n]])$ is isomorphic to a formal scheme of this form (see \cite[Proposition 5.30]{Car21}.
\end{example}

\section{Jet Spaces}
\subsection{Jets of Schemes}\label{algjet}
We recall the construction and state basic facts on jet spaces, following \cite{AG05}.

\begin{definition}
Let $X$ be a scheme of finite type over $\ccx$, and $m\in \nn_0$. The \emph{scheme of $m$-jets} of $X$ is the scheme $J_m(X)$ satisfying for every $\ccx$-algebra $A$ the functorial relation
\[
\Hom(\Spec (A),J_m(X))\cong \Hom(\Spec (A[t]/(t^{m+1})),X).
\]
\end{definition}
These bijections describe the functor of points of $J_m(X)$.

\begin{proposition}\thlabel{jetpt}
For every scheme $X$ of finite type over $\ccx$, and every $m\in \nn_0$, $J_m(X)$ exists and is a scheme of finite type over $\ccx$.
\end{proposition}

\begin{proof}[Proof: \nopunct]
\sloppy Suppose that $X=\Spec (\ccx[x_1,\ldots,x_n]/(f_1,\ldots,f_r))$ is affine; any morphism defining a jet corresponds to a ring homomorphism
 \[
\phi: \ccx[x_1,\ldots,x_n]/(f_1,\ldots,f_r)\rightarrow \ccx[t]/(t^{m+1}).
\]
This map is determined by setting $u_i=\phi(x_i)=\sum_{j=0}^m a_{ij}t^j$ such that $f_l(u_1,\ldots,u_n)\in (t^{m+1})$ for each $l$. Equating co-efficients gives a system of polynomial constraints, equal to formal derivations of the $f_l$, which determine $J_m(X)$.

By \cite[Lemma 2.3]{AG05} the general case follows by gluing affine charts.
\end{proof}

By the definition, $J_m(X)$ is unique up to a canonical isomorphism.
If $m>p$, there is a canonical projection $\pi_{m,p}:J_m(X)\rightarrow J_p(X)$. Furthermore, we have $\pi_{m,p}\circ \pi_{q,m}=\pi_{q,p}$.

$J_0(X)=X$, and so we have projections $\pi_m=\pi_{m,0}:J_m(X)\rightarrow X$.

$J_1(X)$ is canonically isomorphic to the tangent bundle $TX$.

\begin{definition}
Let $X$ be a scheme, $m\in\nn_0$ and $x\in X$. The scheme of $m$-jets of $X$ above $x$ is the fibre of $J_m(X)$ above $x$, denoted $J_m(X,x)$.
\end{definition}

\begin{definition}
Let $X$ be a scheme of finite type over $\ccx$, and $m\in\nn_0$. The \emph{$m$-jets} of $X$ are the closed points of the scheme of $m$-jets
\[
J_m(X)(\ccx)=\Hom(\Spec(\ccx[t]/(t^{m+1})),X).
\]
\end{definition}

\begin{remark}
By abuse of notation, we sometimes write $J_m(X)$ for $J_m(X)(\ccx)$.
\end{remark}

\begin{remark}
Let $X=\vv(I)\subset \aaf^n$ be an affine scheme. Then $J_m(X)\subset J_m(\aaf^n)=\aaf^{(m+1)n}$. We have a system of algebraic co-ordinates 
\[
(a_{ij}\mid 1\leq i\leq n,0\leq j\leq m)
\]
on the jet space.
If $n=2$, we write $a_j$ for $a_{1j}$ and $b_j$ for $a_{2j}$.
\end{remark}

\begin{example}\thlabel{jetex}
Let $C=\vv(xy)\subset \aaf^2$. The jet spaces $J_m(C,0)$ are the union of the co-ordinate hyperplanes
\begin{multline}
a_1=a_2=\cdots=a_{m-1}=0, a_1=\cdots=a_{m-2}=b_1=0,\nonumber \\ 
 \ldots, a_1=b_1=\cdots=b_{m-2}=0, b_1=\cdots=b_{m-1}=0.
\end{multline}

Indeed, as discussed in the proof of \thref{jetpt}, we set $x=\sum a_it^i, y=\sum b_i t^i$. Equating co-efficients so that $xy\in (t^{m+1})$, we have $J_2(C,0)=\vv(a_1 b_1)$. Also, we have $J_1(C,0)=\aaf^2$.

If the result holds for $m$, the extra equation to define the $(m+1)$-jets, that is, the co-efficient of $t^{m+1}$ in the new co-ordinates, is 
\[
(m+1)(a_1 b_m+ a_2 b_{m-1}+\cdots + a_{m-1} b_2+ a_m b_1).
\]
Each string of equalities kills every summand except $b_{j+1} a_{m-j}, 0\leq j \leq m-1$, and so is appended by either $b_{j+1}=0$ or $a_{m-j}=0$. Replacing $m$ with $m+1$, we see that this gives the equations of the next jet space, so the result follows by induction.
\end{example}

\begin{remark}
The mapping $X\mapsto J_m(X)$ is functorial: For any morphism $f:X\rightarrow Y$, and any $m\in\nn$, there is an induced morphism $f_m:J_m(X)\rightarrow J_m(Y)$ given by $\tau\mapsto f\circ\tau$.

If $p<m$, the diagram
\[
\begin{tikzcd}
J_m(X) \arrow[r, "f_m"] \arrow[d, "\pi_{m,p}^X"] & J_m(Y) \arrow[d, "\pi_{m,p}^Y"] \\
J_p(X) \arrow[r, "f_p"] & J_p(Y)
\end{tikzcd}
\]
commutes.
\end{remark}

Taking $X$ to be a closed subscheme of $Y$, and $f$ the inclusion map, we see that $J_m(X)\subset J_m(Y)$.

\begin{proposition}\cite[Proposition 6.9]{Car21}
Let $(Y_i)_{i\in\icurl}$ be a family of closed subschemes of a scheme $X$ of finite type over $\ccx$. Then $J_m(\bigcap_{i\in\icurl} Y_i)=\bigcap_{i\in\icurl} J_m(Y_i)$ for each $m$.
\end{proposition}

\begin{proposition}\thlabel{jetsurj}
Let $f:X\rightarrow Y$ be a proper birational morphism of smooth varieties. Then for each $m\in\nn$, the induced morphism $f_m:J_m(X)\rightarrow J_m(Y)$ is surjective.
\end{proposition}

\begin{proof}[Proof: \nopunct]
See \cite[Proposition 6.18]{Car21}
\end{proof}

\begin{proposition}[Change of Variables]\thlabel{jetco}
Let $W\subset X$ be a closed subscheme, and $\pi:X'\rightarrow X$ a morphism of schemes. For each $m\in \nn$, $J_m(\pi^{-1}(W))=\pi_m^{-1}(J_m(W))$.
\end{proposition}

\begin{proof}[Proof: \nopunct]
See \cite[Proposition 6.20]{Car21}.
\end{proof}

\begin{corollary}\thlabel{pfjet}
$\pi_m(J_m(W))\subset J_m(\pi(W))$, with equality if $\pi_m$ is surjective.
\end{corollary}

\begin{remark}
We can also define the jets of a complex space as follows: For a complex space $X$ and a point $x\in X$, the space of $m$-jets of $X$ above $x$, denoted $J_m^{an}(X,x)$, is the set of equivalence classes of germs of holomorphic maps $f:(\ccx,0)\rightarrow (X,x)$, where $f\sim g$ if $f^{(i)}(0)=g^{(i)}(0)$ for all $0\leq i\leq m$. We define $J_m^{an}(X)=\bigcup_{x\in X}J_m^{an}(X,x)$.

We see that if $X$ is algebraic, this space is equal to the space of $m$-jets defined above.
\end{remark}

\subsection{Jets of Formal Schemes}

\begin{proposition}\thlabel{jetform}
Let $X=\varinjlim X_{\lambda}$ be a formal scheme. Then for each $m\in\nn$ and $x\in X$, $J_m(X,x)=\bigcup_{\lambda\in \Lambda} J_m(X_{\lambda},x)$. Furthermore, the jets are independent of the limit presentation.
\end{proposition}

Proof: omitted.

\begin{corollary}
Let $X=\varinjlim X_{\lambda}$ be a direct limit of formal schemes. Then for each $m\in\nn$ and $x\in X$, $J_m(X,x)=\bigcup_{\lambda\in \Lambda} J_m(X_{\lambda},x)$.
\end{corollary}

\begin{remark}
The method for determining the defining equations of $J_m(X)$ given in \thref{jetpt} also holds when $X$ is a complex space or a formal scheme defined by formal power series. The change of variables formula also holds in these cases.
\end{remark}

\begin{proposition}\thlabel{fibres}
Let $X$ be either a complex space or a formal scheme, and let $x\in X$. Then for each $m\in\nn$, the fibre of the jet space $J_m(X,x)$ is an affine scheme.

Furthermore, if $f:X\rightarrow Y$ is a (formal) isomorphism of complex spaces, then the induced isomorphisms $J_m(X,x)\rightarrow J_m(Y,f(x))$ are algebraic.
\end{proposition}

\begin{proof}[Proof: \nopunct]
See \cite[Proposition 6.26, Corollary 6.30, Proposition 6.31]{Car21}.
\end{proof}

\begin{example}
Let $X$ be a smooth variety, $Z\subset X$ a closed subscheme, and $z\in Z$. Consider the formal completion $\hat{X}_Z$ along $Z$. Then $J_m(\hat{X}_Z,z)=J_m(X,z)$.

Indeed, $\hat{X}_Z$ is the formal direct limit of the schemes $Z^i$, so by \thref{jetform}, $J_m(\hat{X}_Z,z)=\bigcup_{i\in\nn}J_m(Z^i,z)$. As $Z^i$ has degree $\geq i$, its jets of lower orders are the full affine space, equal to $J_m(X,z)$.
\end{example}

\begin{lemma}\thlabel{jetstruc}
Let $X$ be a smooth variety, and let $Y_1,Y_2$ be two formal subschemes of $X$ with the same underlying topological space. Suppose that $J_m(Y_1)\subset J_m(Y_2)$, for all $m\in\nn$. Then $Y_1\subset Y_2$.
\end{lemma}

\begin{proof}[Proof: \nopunct]
Let $\pi_m^X$ be the truncation map $J_m(X)\rightarrow X$. Then $\pi_m^X\restn{Y_i}=\pi_m^{Y_i}:J_m(Y_i)\rightarrow Y_i$. We have $\pi_m^X(J_m(Y_1))\subset\pi_m^X(J_m(Y_2))$, and so, as $Y_1$ and $Y_2$ have the same underlying topological space, we have $\pi_m^{Y_1}(J_m(Y_1))\subset\pi_m^{Y_2}(J_m(Y_2))$. This holds for all $m\in\nn$, so $Y_1\subset Y_2$.
\end{proof}

\section{Foliations}\label{Foliation}
Throughout this section we let $X$ be a complex manifold.
\subsection{Foliations by Vector Fields}
\begin{definition}
A subsheaf $\fcurl\subset \Tcal_X$ is called a \emph{foliation} if it is saturated, that is, $\Tcal_X/\fcurl$ is torsion-free, and is integrable, that is, $[\fcurl,\fcurl]\subset \fcurl$, where the Lie bracket is defined on the sections of the sheaf.
\end{definition}

\begin{remark}
The condition of being saturated is not preserved under many operations, for example pullback. Hence we sometimes deal with unsaturated foliations. If $\fcurl$ is an unsaturated foliation, we define its saturation $\sat(\fcurl)$ to be the smallest saturated sheaf containing it. It is the kernel of the morphism 
\[
\Tcal_X \rightarrow \Tcal_X/\fcurl \rightarrow (\Tcal_X/\fcurl)/ \Tors(\Tcal_X/\fcurl).
\]
\end{remark}

\begin{definition}
Let $\fcurl$ be a foliation on $X$. The \emph{singular locus} of $\fcurl$, denoted $\Singr(\fcurl)$, is the locus of points of $X$ around which $\fcurl$ is not locally free. It is a complex subspace of $X$ of codimension at least $2$.
\end{definition}

\begin{remark}
Let $\fcurl$ be a foliation on $X$. The tangent sheaf $\Tcal_X$ is reflexive, and so, as $\fcurl$ is saturated, \cite[Lemma II.1.1.16]{OSS80} gives that $\fcurl$ is normal, that is, for any open $U\subset X$ and any analytic subset $A\subset U$ of codimension at least $2$, the restriction map $\fcurl(U)\rightarrow\fcurl(U\setminus A)$ is an isomorphism. By the saturation property $A=\Singr\fcurl$ has codimension at least $2$, so in particular, for all open $U\subset X$, we have $\fcurl(U)\cong \fcurl(U\setminus A)$. So the global behaviour of $\fcurl$ is defined by its behaviour on the smooth locus.
\end{remark}

\begin{remark}
By Frobenius' theorem (see \cite{Fro1877}), we can write $X\setminus \Singr\fcurl$ as the disjoint union of connected submanifolds $L_{\alpha}$ -- the leaves of the foliation -- where for each $x\in X$, there is a system of local holomorphic co-ordinates $x_1,\ldots,x_n$ on an open neighbourhood $U\ni x$, such that the components of $U\cap L_{\alpha}$ can be written as
\[
x_{r+1}=\mu_{r+1},\ldots, x_n=\mu_n,
\]
with the $\mu_i$ constant.
\end{remark}

\subsection{Foliations by $1$-forms}
\begin{theorem}\thlabel{dualfol}
Let $X$ be a complex manifold. Then there is a one-to-one correspondence between saturated subsheaves of the sheaf of $1$-forms $\Omega=\Omega^1_X$ and saturated subsheaves of the tangent sheaf $\Tcal=\Tcal_X$.
\end{theorem}

\begin{proof}[Proof: \nopunct]
See \cite[Theorem 8.9]{Car21}
\end{proof}

\begin{definition}
Let $\gcurl\subset \Omega^1_X$ be a subsheaf. On some open subset $U\subset X$, $\gcurl(U)$ is generated by $1$-forms $\omega_1,\ldots, \omega_r$, which generate an ideal in $\Omega(U)$, the ring of all differential forms, with $\wedge$ as multiplication. $\gcurl$ is said to be integrable if this ideal is closed under the exterior derivative.
\end{definition}

\begin{lemma}
Let $\gcurl\subset \Omega^1_X$ be a saturated subsheaf. Then $\gcurl$ is integrable if and only if the subsheaf $\fcurl \subset \Tcal_X$ it corresponds to by \thref{dualfol} is integrable. 
\end{lemma}

\begin{proof}[Proof: \nopunct]
See \cite[Lemma 8.11]{Car21}.
\end{proof}

Thus any foliation $\fcurl$ is defined uniquely by a saturated subsheaf of the sheaf of $1$-forms $\Omega$, and so the foliation is locally defined as the integrable distribution of vector fields annihilated by a collection of $1$-forms. A foliation can be equivalently defined either in terms of vector fields or in terms of $1$-forms, as convenient.

\begin{lemma}
Let $\gcurl\subset \Omega^1_X$ be a subsheaf given locally by $\omega=b_1dx_1+\cdots +b_ndx_n$, where the $b_i$ are holomorphic.

(1) $\gcurl$ is integrable if and only if $\omega\wedge d\omega=0$.

(2) $\gcurl$ is saturated if and only if $\gcd(b_1,\ldots,b_n)=1$.

(3) $\Singr\gcurl=\vv(b_1,\ldots,b_n)$.
\end{lemma}

\begin{proof}[Proof: \nopunct]
See \cite[Lemma 8.12, Lemma 8.14, Lemma 8.16]{Car21}.
\end{proof}

\begin{corollary}
Let $\gcurl\subset \Omega^1_X$ be a subsheaf of corank $1$. Then $\gcurl$ is integrable if and only if its saturation is.
\end{corollary}

\subsection{Pullback Foliations}

Let $f:X\rightarrow Y$ be a holomorphic map of complex manifolds, and let $\fcurl$ be a foliation on $Y$ locally generated by $\omega_1,\ldots, \omega_r$. Then the pullback of these forms $f^*\omega_1,\ldots, f^*\omega_r$ generate a (possibly unsaturated) foliation on $X$, which we denote by $f^{-1}(\fcurl)$.

\begin{definition}
Let $\fcurl$ be a foliation on $X$, and $V\subset X$ a reduced, irreducible complex subspace. The \emph{restriction} $\fcurl\restn{V}$ of $\fcurl$ to $V$ is the pullback of $\fcurl$ along the inclusion map $\iota:V\rightarrow X$.
\end{definition}

\begin{definition}
Let $X$ and $Y$ be manifolds, and $\fcurl$ a foliation on $X$. Let $\gcurl$ be the pullback of $\fcurl$ along the projection $X\times Y\rightarrow X$. Then $\gcurl$ is called the \emph{cylinder} over $\fcurl$.
\end{definition}

\begin{example}
Let $X=\aaf^n$, and $\fcurl$ be given by the form $\omega=f_1dx_1+\cdots +f_kdx_k, k<n$, where the $f_i$ are functions of $x_1,\ldots,x_k$ only. Then $\fcurl$ is the cylinder over the foliation on $\aaf^k$ given by $\omega$.
\end{example}

\section{Jet Spaces of Foliations}\label{Jetfol}
Throughout, we let $X$ be a complex manifold.

\subsection{Basic Definitions}\label{foljet}
\begin{definition}
Let $\fcurl$ be a foliation on $X$, locally generated by $1$-forms $\omega_1,\ldots,\omega_r$. The \emph{jet space} of $\fcurl$ is defined as
\[
J_m(\fcurl)=\{\tau\in J_m(X) \mid \tau^*(\omega_l)=0, 1\leq l\leq r\},
\]
where $\tau^*(\omega_l)$ means the pullback along $\tau$ interpreted as a morphism 
\[
\Spec (\ccx[t]/(t^{m+1}))\rightarrow X.
\]
\end{definition}

If on some open subset of $X$ we have local co-ordinates $x_1,\ldots,x_n$, then, following the construction of the jet space of an affine scheme (see \thref{jetpt}), we define a morphism $\tau:\Spec (\ccx[t]/(t^{m+1}))\rightarrow X$ by setting $x_i = \sum_{j=0}^m a_{ij}t^j$, and by pullback we have the differential $dx_i$ mapped to $\sum_{j=1}^m ja_{ij}t^{j-1}dt$. The jet space $J_m(\fcurl)$ is again defined as the vanishing locus of the polynomial constraints imposed on the $a_{ij}$ to ensure that $\tau^*(\omega)\in (t^m)dt$ for each of the $1$-forms $\omega$ defining the foliation, noting that as $t^{m+1}=0$, we have $t^m dt=0$. We thus see that $J_m(\fcurl)$ is a subscheme of $J_m(X)$.

As $J_m(\fcurl)\subset J_m(X)$, we can look at its fibres $J_m(\fcurl,x)=J_m(X,x)\cap J_m(\fcurl)$, again constructed in the same way as with jet spaces of schemes.

\begin{proposition}[Change of Variables for Foliations]\thlabel{folco}
Let $f:X\rightarrow Y$ be a map of complex manifolds, and $\fcurl$ a foliation on $Y$. Then $J_m(f^{-1}(\fcurl))=f_m^{-1}(J_m(\fcurl))$.
\end{proposition}

\begin{proof}[Proof: \nopunct]
See \cite[Proposition 9.2]{Car21}
\end{proof}

\begin{corollary}\thlabel{folfib}
The fibres of the jet spaces of foliations, and the morphisms between them induced by (formal) isomorphisms of the ambient space, are algebraic. In particular, formally equivalent vector fields yield isomorphic jets.
\end{corollary}

\begin{proof}[Proof: \nopunct]
If $\fcurl$ is a foliation on $X$, then $J_m(\fcurl,x)\subset J_m(X,x)$. The result then follows using the same arguments as in \thref{fibres}.
\end{proof}

\begin{lemma}\thlabel{firstint}
Let $\fcurl$ be a foliation given by $1$-forms $\omega_1,\ldots, \omega_r$, each of which has an algebraic first integral: that is, for each $l$, $\omega_l=dg_l$ for some polynomial $g_l$. Then for any point $x\in\vv(g_1,\ldots, g_r)$, $J_m(\fcurl,x)=J_m(\vv(g_1,\ldots, g_r),x)$.
\end{lemma}

\begin{proposition}
Let $\fcurl$ be a foliation on $X$, and $V$ be a reduced, irreducible complex subspace. Set $\gcurl=\fcurl\restn{V}$. Then $J_m(\gcurl)=J_m(\fcurl)\cap J_m(V)$, for all $m$.
\end{proposition}

Proof: omitted.

\begin{example}\thlabel{reducedex}
We now let $X=\aaf^2$, and consider the foliation $\fcurl_1$ given by the $1$-form $\omega_1=ydx+xdy$, which has a single singular point at the origin. To calculate the $m$-jets above the origin (the singular point), we set
\[
x=a_1t+a_2t^2+\cdots+a_mt^m;
y=b_1t+b_2t^2+\cdots+b_mt^m,
\]
and so 
\[
dx=(a_1+2a_2t+\cdots+ma_mt^{m-1})dt;
dy=(b_1+2b_2t+\cdots+mb_mt^{m-1})dt.
\]
By equating co-efficients, we find the $a_i,b_i$ such that the image of $\omega$ under this morphism lies in $(t^m)dt$.

This foliation has a first integral $xy$. Then by \thref{firstint}, for each $m$, $J_m(\fcurl_1,0)=J_m(C,0)$. These jets were calculated in \thref{jetex}.

The foliation $\fcurl_2$ given by $\omega_2=ydx-x^2 dy$ has the same jets above the origin, so the jets at the singular locus are not sufficient to determine the foliation.
\end{example}

\subsection{Tangent Schemes}
\begin{definition}\thlabel{tangent}
Let $\fcurl$ be a foliation on a complex manifold $X$, and $C\subset X$ a complex subspace.

$C$ is \emph{weakly tangent} to $\fcurl$ if $J_1(C)\subset J_1(\fcurl)$.

$C$ is \emph{strongly tangent} to $\fcurl$ if $J_m(C)\subset J_m(\fcurl)$, for all $m\in\nn$.

$C$ is \emph{fully tangent} to $\fcurl$ if $J_m(C)=J_m(\fcurl)\restn C$, for all $m\in \nn$.

A complex subspace $C\subset X$ is a \emph{solution} for the foliation $\fcurl$ if locally its defining equations $g_1,\ldots,g_k$ solve all the $1$-forms $\omega_1,\ldots,\omega_r$ that determine $\fcurl$; that is, $\omega_l\restn C\in (dg_1,\ldots,dg_k)\restn C\subset \Omega^1_X\restn C$ for each $l$.
\end{definition}

\begin{remark}
We can also define all these notions of tangency if $C$ is instead a formal subscheme of $X$. (A formal scheme $C=\varinjlim Y_{\lambda}$ can be viewed as a subscheme of a complex manifold $X$ if we identify each $Y_{\lambda}$ with its associated complex space, and take the formal direct limit of complex subspaces of $X$). For a formal scheme to be a solution, we assume it is given by some collection of formal power series.
\end{remark}

\begin{remark}
The notions of tangency found in the literature, for example being weakly tangent or a solution, though elementary, fail to encapsulate the behaviour of the foliation at the singular locus. The notion of being strongly tangent rectifies this by allowing us to consider the tangency of subspaces with non-reduced structure.
\end{remark}

\begin{remark}
If a subspace $C\subset X$ is reduced, then it is strongly tangent if it is a solution (see \thref{redstt} below). Conversely, if $C$ is reduced and strongly tangent, then it is a solution if it is either contained in the smooth locus (\cite[Proposition 9.20]{Car21}), or irreducible and of the same dimension as the leaves (\cite[Corollary 9.22]{Car21}).
\end{remark}

\begin{example}
Let $X=\aaf^2$, $\omega=ydx+xdy$. The leaves of the resultant foliation are of the form $\{xy=a\}$; these are clearly solutions for the  foliation. All points of the plane, viewed as subspaces, are also strongly tangent -- we see that any jet of a point is a constant function, as so pulls back to zero any $1$-form. 

Tangent schemes need not be reduced: let $C$ be the origin counted with multiplicity $2$---this is defined by the equations $x^2, xy, y^2$. As $d(xy)=ydx+xdy$, we see that $C$ is indeed a solution of the foliation.

 However the origin counted with multiplicity $3$ is not: This is defined by $x^3, x^2y,xy^2,y^3$, and so the space of $2$-jets is all of $\aaf^4$, which is not contained in $J_2(\fcurl,0)$, as calculated in \thref{reducedex}.
\end{example}

\begin{remark}
It follows from the definitions that all notions of tangency (as described in \thref{tangent}) are locally defined. Furthermore, if $C\subset X$ is weakly tangent (respectively, strongly tangent, respectively, a solution), then any subspace of $C$ is also weakly tangent (respectively, strongly tangent, respectively, a solution).
\end{remark}

\begin{remark}
Note that the union of strongly tangent schemes is not in general strongly tangent: Let $X=\aaf^2$, $\fcurl$ be given by $\omega=ydx-xdy$, $C_1=\vv(xy)$, and $C_2=\vv(y-x)$. $C_1$ and $C_2$ are both strongly tangent, but $C_1 \cup C_2$ is not: It has $x=t+t^2+t^3, y=t+2t^2+t^3$ as a $3$-jet over the origin, which does not lie in $J_3(\fcurl,0)$.
\end{remark}

\begin{proposition}\thlabel{tanginv}
All notions of tangency are invariant under co-ordinate changes.
\end{proposition}

\begin{proof}[Proof: \nopunct]
The proof follows from the definitions and the change of variables formulas. For the details of the calculations, see \cite[Proposition 9.14]{Car21}.
\end{proof}

\begin{corollary}\thlabel{pftang}
Suppose $\pi:X'\rightarrow X$ is a proper birational morphism (for example, a sequence of blow-ups in smooth centres), and suppose, for some $C\subset X$, that $\pi^{-1}(C)$ is strongly tangent (respectively, weakly tangent, respectively, fully tangent) to $\pi^{-1}(\fcurl)$. Then $C$ is strongly tangent (respectively, weakly tangent, respectively, fully tangent) to $\fcurl$.

Similarly, if $C\subset X'$ is strongly tangent (respectively, weakly tangent, respectively, fully tangent) to $\pi^{-1}(\fcurl)$, then $\pi(C)$ is strongly tangent (respectively, weakly tangent, respectively, fully tangent) to $\fcurl$.
\end{corollary}

\begin{proof}[Proof: \nopunct]
Suppose first that $\pi^{-1}(C)$ is strongly tangent to $\pi^{-1}(\fcurl)$. Then for all $m\in\nn$, $J_m(\pi^{-1}(C))\subset J_m(\pi^{-1}(\fcurl))$. By \thref{jetco,folco} we then have $\pi_m^{-1}(J_m(C))\subset \pi_m^{-1}(J_m(\fcurl))$. By \thref{jetsurj}, $\pi_m$ is surjective, and so $J_m(C)\subset J_m(\fcurl)$, for all $m\in\nn$; that is, $C$ is strongly tangent to $\fcurl$.

In the second case, if $C$ is strongly tangent to $\pi^{-1}(\fcurl)$, then for all $m\in\nn$, $J_m(C)\subset J_m(\pi^{-1}(\fcurl))=\pi_m^{-1}(J_m(\fcurl))$, and hence $\pi_m(J_m(C))\subset \pi_m(\pi_m^{-1}(J_m(\fcurl)))$. By \thref{jetsurj}, $\pi_m$ is surjective, and so by \thref{pfjet} we have $J_m(\pi(C))\subset J_m(\fcurl)$, for all $m\in\nn$. So $\pi(C)$ is strongly tangent to $\fcurl$.

The cases for weakly tangent and fully tangent are proved in the same way.
\end{proof}

\begin{lemma}\thlabel{redstt}
Let $\fcurl$ be a foliation on $X$ and $C\subset X$ a reduced solution for $\fcurl$. Then $C$ is strongly tangent to $\fcurl$.
\end{lemma}

\begin{proof}[Proof: \nopunct]
Let $g_1,\ldots,g_k$ be the defining equations of $C$ in a local co-ordinate system, and $\tau\in J_m(C)$. Then $g_r\circ\tau\in(t^{m+1})$, and $\tau^*dg_r=d(g_r\circ\tau)=0\in\Omega_{\Spec(\ccx[t]/(t^{m+1}))}$. As $C$ is a solution for $\fcurl$, all the $1$-forms $\omega_l$ determining $\fcurl$ satisfy $\omega_l\restn{C}\in (dg_1,\ldots,dg_k)\restn{C}$, hence $\omega_l\restn{C}$ pulls back to $0$ too. It follows that $C$ is strongly tangent to the (unsaturated) foliation given by $\omega_l\restn{C}$.

As $C$ is tangent to a foliation on $C$, the saturation of this foliation is given by the zero form; as $C$ is reduced, it is strongly tangent to this saturation, which is equal to the saturation of the pullback foliation $\fcurl\restn{C}$, whose jets are contained in the jets of $\fcurl$. So $C$ is strongly tangent to $\fcurl$.
\end{proof}

\begin{proposition}\thlabel{unsat}
Let $\fcurl$ be a foliation given by $1$-forms $\omega_1,\ldots,\omega_r$, and $\fcurl'$ the unsaturated foliation given by $f\omega_1,\ldots, f\omega_r$, for some holomorphic function $f$. If $C=\vv(g_1,\ldots,g_k)$ is tangent to $\fcurl$, for any of the notions of tangency given in \thref{tangent}, then $C'=\vv(fg_1,\ldots,fg_k)$ is tangent to $\fcurl'$.
\end{proposition}

\begin{proof}[Proof: \nopunct]
For all of the worked calculations, see \cite[Proposition 9.23]{Car21}.
\end{proof}

\begin{corollary}\thlabel{tangform}
Let $X=\varinjlim X_{\lambda}$ be a direct limit of formal schemes. If each $X_{\lambda}$ is strongly tangent to a foliation $\fcurl$, then $X$ is also strongly tangent to $\fcurl$.
\end{corollary}

\begin{proof}[Proof: \nopunct]
$J_m(X,x)=\bigcup_{\lambda\in \Lambda} J_m(X_{\lambda},x)\subset J_m(\fcurl,x)$, by strong tangency of the $X_{\lambda}$. Hence $X$ is strongly tangent.
\end{proof}

\subsection{Separatrices and Dicriticality}\label{sepdicrit}
\begin{definition}
Let $\fcurl$ be a singular foliation on a manifold $X$. A \emph{separatrix} is a reduced, irreducible complex subspace of $X$, of dimension equal to that of the leaves of $\fcurl$, which intersects the singular locus and is strongly tangent to the foliation.
\end{definition}

\begin{remark}
A separatrix is in fact the closure of a leaf of the foliation that extends holomorphically through the singular locus, and so is a solution for the foliation.

We also allow for \emph{formal separatrices}, which are formal subschemes of $X$ strongly tangent to the foliation (and which are reduced, irreducible, of leaf dimension, and which intersect the singular locus).
\end{remark}

\begin{example}
Consider the foliations on $X=\aaf^2$ given by the following $1$-forms:

(1) $\omega =ydx+xdy$. This has separatrices $\{x=0\}$ and $\{y=0\}$.

(2) $\omega =ydx-xdy$. Here every line through the origin is a separatrix.

(3) $\omega =ydx-(x+y)dy$. The only separatrix is $\{y=0\}$. The other leaves of the foliation cannot be extended holomorphically through the origin, so do not satisfy the definition.

(4) $\omega =(y-x)dx-x^2 dy$. This has one convergent (that is, not formal) separatrix, $\{x=0\}$. It also has a formal separatrix, given by the formal power series $y=\sum_{k=0}^{\infty} k! x^{k+1}$.
\end{example}

We now give some results on the existence of separatrices. We henceforth assume that $X$ is quasi-compact, and the foliation is codimension-$1$, and generated by an algebraic $1$-form. (So in particular the singular locus has finitely many irreducible components.)

\begin{theorem}\cite{CS82}
Suppose $\dimn X=2$. Then any foliation on $X$ has a separatrix.
\end{theorem}

The same is not true in higher dimensions: Let $X=\aaf^3$, and, for $m\geq 2$, let $\fcurl_m$ be the foliation given by
\[
\omega_m=(x^m y-z^{m+1})dx+(y^m z-x^{m+1})dy+(z^m x-y^{m+1})dz.
\]
Then none of the foliations $\fcurl_m$ have any separatrices at the origin. (See \cite{Jou79}).

To proceed, we introduce the following notion:
\begin{definition}
A codimension-$1$ foliation $\fcurl$ on $X$ is said to be \emph{dicritical} if there exists a sequence of blow-ups in smooth centres, where the centre of each blow-up is contained in the singular locus, after which a component of the exceptional divisor is transversal to the transformed foliation (that is, is not a separatrix).

Otherwise $\fcurl$ is called \emph{non-dicritical}.
\end{definition}

\begin{example}
Let $\fcurl$ be the foliation on $\aaf^2$ given by $\omega =ydx-xdy$. Then $\fcurl$ is dicritical. Indeed, blowing up at the origin we get the form $-x^2 dv$, which describes a foliation whose saturation is transversal to the exceptional divisor.
\end{example}

\begin{remark}
By setting $n=2$ in \cite[Theorem 4]{CM92}, we see that a foliation on a surface is dicritical if and only if it has infinitely many separatrices.
\end{remark}

\begin{example}
The foliations $\fcurl_m$ on $\aaf^3$ defined above are dicritical. Indeed, blowing up at the origin we get the form
\[
x^{m+2}((u^m v-1)du+(v^m-u^{m+1})dv).
\]
The exceptional divisor is seen to be transversal to the saturated foliation.
\end{example}

\begin{theorem}\cite[Theorem 5]{CM92}
Any non-dicritical foliation has a separatrix.
\end{theorem}

In \cite{CM92} the theorem is given in terms of germs. However, if the germ of the foliation has a germ of a separatrix, then expanding from a formal to an open neighbourhood of the origin, we see that the hypersurface will still be a separatrix to the full foliation.

\begin{proposition}\thlabel{finsep}
A non-dicritical foliation on a quasi-compact manifold has only finitely many separatrices.
\end{proposition}

\begin{proof}[Proof: \nopunct]
See \cite[Proposition 9.38]{Car21}.
\end{proof}

\section{Singularities of Codimension-$1$ Foliations}

In this section, we let $X$ be a quasi-compact complex manifold of arbitrary dimension, and let $\fcurl$ be a codimension-$1$ foliation on $X$ given by an algebraic $1$-form. We study the behaviour of the singular loci.

\subsection{Preliminaries}
\begin{example}\thlabel{jetex2}
Let $C=\vv(x_1\cdots x_n)\subset \aaf^n, n\geq 2$, and suppose $m\geq n$. Then, setting $x_i=\sum a_{ij}t^j$, we have
\[
J_m(C,0)=\bigcup_{j_1+\cdots +j_n=m-n+1} \bigcap_{1\leq i\leq n, 1\leq j\leq j_i} \{a_{ij}=0\}.
\]

\end{example}

\begin{proof}[Proof: \nopunct]
We prove by induction on $n$. The base case $n=2$ is proved in \thref{jetex}.

For $n\geq 3$ and $m\geq n$, the co-efficient of $t^m$, which vanishes for jets of order $m$ and higher, is
\[
y_{n-1}a_{n,m-n+1}+y_n a_{n,m-n}+y_{n+1}a_{n,m-n-1}+\cdots +y_{m-1}a_{n,1},
\]
where $y_j$ is the co-efficient of $t^j$ obtained from $y=x_1\cdots x_{n-1}$. (Note that $y_1=\cdots =y_{n-2}=0$.)

Then $J_m(C,0)$ is given by the union of the sets cut out by the equations
\begin{multline}
a_{n,1}=a_{n,2}=\cdots =a_{n,m-n+1}=0, a_{n,1}=\cdots=a_{n,m-n}=y_{n-1}=0, \ldots, \nonumber \\
a_{n,1}=y_{n-1}=y_n=\cdots=y_{m-2}=0, y_{n-1}=y_n=\cdots =y_{m-1}=0,
\end{multline}
by the same argument as in \thref{jetex}.

By the induction hypothesis, the equations $y_{n-1}=\cdots=y_p=0$ give
\[
\bigcup\{a_{11}=\cdots=a_{1j_1}=\cdots=a_{n-1,1}=\cdots =a_{n-1,j_{n-1}}=0\},
\]
with the union taken over indices with $j_1+\cdots j_{n-1}=p-(n-1)+1=p-n+2$.

Now we have $j_n=m-1-p$, and so $j_1+\cdots +j_n=m-n+1$, as required.
\end{proof}

Now let $x=(x_1,\ldots,x_n)$ be a non-zero point of $C$; let $I$ be the set of indices of its non-zero entries. Let $C'=\vv(\prod_{i\notin I} x_i)$. Then $J_m(C,x)=J_m(C',0)$.

\begin{definition}
Let $\fcurl$ be a foliation on a complex manifold $X$ given locally by a $1$-form $\omega$, (so in particular $\omega$ generates a saturated sheaf, and $\omega\wedge d\omega=0$), and let $x\in X$. The \emph{dimensional type} $\tau(\fcurl,x)$ is the codimension in $T_xX$ of $\Dcal_{\fcurl}(x)=\{\Xcal(x)\mid \omega(\Xcal)=0\}$, where $\Xcal$ is a germ of a vector field at $x$.
\end{definition}

\begin{proposition}\thlabel{dimtyp}
Let $\fcurl$ be a codimension-$1$ foliation on $X$, and let $x\in X$. The dimension type $\tau(\fcurl,x)$ is the minimal number of formal co-ordinates needed to write a generator $\omega$ of the foliation in a neighbourhood of $x$.
\end{proposition}

\begin{proof}[Proof: \nopunct]
See \cite[Proposition 10.7]{Car21}.
\end{proof}

\subsection{Simple Singularities}
Let $\fcurl$ be a codimension-$1$ foliation on a complex manifold $X$, let $x\in X$ be a point with dimensional type $t=\tau(\fcurl,x)$, and let $E$ be a normal crossings divisor of $X$ through $x$ with each component tangent to $\fcurl$. 

\begin{proposition}\thlabel{elesst}
In this setting, $E$ has at most $t$ components through $x$.
\end{proposition}

\begin{proof}[Proof: \nopunct]
In a neighbourhood of $x$, $\fcurl$ is a cylinder over a foliation on a $t$-dimensional subspace of $X$. As each component of $E$ is tangent to $\fcurl$, $E$ is a cylinder over an SNC divisor of this subspace, and so has at most $t$ components.
\end{proof}

Taking appropriate holomorphic co-ordinates at $x$, as in \cite[Section 4]{Can97}, there is a subset $A\subset \{1,\ldots,t\}$ such that we can write $E=\vv(\prod_{i\in A}x_i)$, and the foliation is generated by
\[
\omega=\left(\prod_{i\in A}x_i\right)\left(\sum_{i\in A}b_i \frac{dx_i}{x_i}+\sum_{i\in \{1,\ldots,t\}\setminus A} b_i dx_i\right),
\]
where $b_i=b_i(x_1,\ldots,x_t)$ are germs of holomorphic functions without common factor.
Indeed, as the dimensional type is $t$ at $x$, we can write $\omega$ in the first $t$ co-ordinates as $\omega=\sum_{i\in A}f_i dx_i+\sum_{i\notin A}f_i dx_i$. As the components of $E$ are tangent, we have $x_i\mid f_j$, for all $i\in A$ and all $j\neq i$. Setting $b_i=\frac{f_i}{\prod_{j\in A, j\neq i}x_j}$ yields the result.

\begin{proposition}\thlabel{snctang}
Let $\fcurl$ be a foliation on $X$, and let $E$ be an SNC divisor with each component tangent to $\fcurl$. Then $E$ itself is a solution for $\fcurl$.
\end{proposition}

\begin{proof}[Proof: \nopunct]
We choose holomorphic co-ordinates so $E$ and $\fcurl$ are given in the form above. If $E=\vv(\prod_{i\in A}x_i)$, then for each component $\vv(x_i)$ of $E$, we have $\omega\restn {x_i=0}=b_i d(\prod_{i\in A}x_i)\restn {x_i=0}$, and the result follows.
\end{proof}

It follows from \thref{redstt} that $E$ is strongly tangent.

\begin{lemma}
Let $\fcurl$ be a foliation on $X$ given by $\omega$ and let $C\subset X$ be a smooth irreducible hypersurface. Suppose for some point $x\in C$, $J_m(C,x)\subset J_m(\fcurl,x)$, for all $m$. Then $C$ is a solution for $\fcurl$.
\end{lemma}

\begin{proof}[Proof: \nopunct]
See \cite[Lemma 10.14]{Car21}.
\end{proof}

\begin{corollary}\thlabel{snctang2}
The same result holds if $C$ is an SNC divisor with smooth components.
\end{corollary}

\begin{proof}[Proof: \nopunct]
Each component of $C$ satisfies the conditions of the above lemma, so is a solution for $\fcurl$. As $C$ is SNC, it too is a solution by \thref{snctang}.
\end{proof}

\begin{definition}\cite[Section 4]{Can97}
Let $\fcurl$ be a codimension-$1$ foliation on a complex manifold $X$, let $x\in X$ be a point with dimensional type $t=\tau(\fcurl,x)$, and let $E$ be a normal crossings divisor of $X$ through $x$ with each component tangent to $\fcurl$. We write the generator of the foliation in the form
\[
\omega=\left(\prod_{i\in A}x_i\right)\left(\sum_{i\in A}b_i \frac{dx_i}{x_i}+\sum_{i\in \{1,\ldots,t\}\setminus A} b_i dx_i\right),
\]
with the $b_i$ holomorphic.

The \emph{adapted order} is $\nu(\fcurl,E;x)=\min\{\ord_x b_i\}$.

The \emph{adapted multiplicity} $\mu(\fcurl,E;x)$ is the order at $x$ of the ideal generated by
\[
\{b_i\}_{i\in A}\cup \{x_j b_i\}_{i\notin A,j=1,\ldots,n},
\]
that is, the minimum of the orders at $x$ of all the functions in the ideal.
\end{definition}

\begin{remark}
As composition with an invertible holomorphic function does not change the order, the adapted order and adapted multiplicity are independent of the normal form chosen.
\end{remark}

\begin{definition}\cite[Definition 4]{Can97}
In the above situation, $x\in\Singr\fcurl$ is a \emph{pre-simple} singularity of $\fcurl$ adapted to $E$ if and only if one of the following occurs:

$\nu(\fcurl,E;x)=0$;

$\nu(\fcurl,E;x)=\mu(\fcurl,E;x)=1$, and for some $i\in A$, the linear part of $b_i$ does not depend only on $\{x_i\mid i\in A\}$.
\end{definition}

\begin{proposition} \cite[Proposition 12]{Can97}\thlabel{simple}
Let $x\in\Singr\fcurl$ be a pre-simple singularity (that is, pre-simple adapted to some SNC divisor $E$), with $\tau(\fcurl,x)=t$. Then in a formal co-ordinate system $x_1,\ldots,x_n$ at $x$, $\fcurl$ is locally generated by a $1$-form in one of the following normal forms:

(A): $\omega=x_1\cdots x_t(\sum_{i=1}^t \lambda_i \frac{dx_i}{x_i})$, $\lambda_i\in \ccx^*$;

(B): $\omega=x_1\cdots x_t(\sum_{i=1}^k p_i \frac{dx_i}{x_i}+\Psi(x_1^{p_1}\cdots x_k^{p_k})\sum_{i=2}^t \lambda_i \frac{dx_i}{x_i})$, where $1\leq k\leq t$, $1\leq p_1,\ldots,p_k\in \nn$ (we can assume them to have no common factor), $\lambda_i\in\ccx$, with $\lambda_{k+1},\ldots,\lambda_t\in\ccx^*$, and $\Psi\in \hat{\mm}_1$, which we can assume to be non-vanishing except at $0$;

(C): $\omega=x_2\cdots x_t(dx_1-x_1\sum_{i=2}^k p_i \frac{dx_i}{x_i}+x_2^{p_2}\cdots x_k^{p_k}\sum_{i=2}^t \lambda_i \frac{dx_i}{x_i})$, where $k\geq 2$, $p_2,\ldots,p_k \in \nn$, and $\lambda_i\in\ccx$, with $\lambda_{k+1},\ldots,\lambda_t\in\ccx^*$.
\end{proposition}

These three cases are mutually exclusive; case (A) can be seen as case (B) with $k=0$.

In cases (A) and (B), we can take $E=\vv(x_1\cdots x_t)$; in case (C), we can take $E=\vv(x_2\cdots x_t)$.

A $1$-form $\omega$ in one of these normal forms is said to be of the form (A), (respectively, (B) or (C)). A singular point $x$ of a foliation $\fcurl$ is said to be of the form (A), (respectively, (B) or (C)), if in a neighbourhood of $x$, $\fcurl$ is generated by such a $1$-form.

\begin{definition}
A pre-simple singularity is \emph{simple} if we are in the case (A) or (B) above, and the tuple $(\lambda_{k+1},\ldots,\lambda_t)$ is non-resonant in the sense that for all maps $\phi:\{k+1,\ldots,t\}\rightarrow \nn_0$, not constantly zero, $\sum_{j=k+1}^t \phi(j)\lambda_j \allowbreak\neq 0$.
\end{definition}

\begin{remark}
In the case $\dimn X=2$, simple singularities are also called reduced singularities. These are usually defined as follows: If, in a neighbourhood of a singular point $x$, $\fcurl$ is given by a vector field $\Xcal$ whose linear part has eigenvalues $\lambda_1, \lambda_2$, then $x$ is a reduced singularity if the $\lambda_i$ are not both zero, and their ratio is not a positive rational number.

In \cite[Section 4]{IlYa08} and \cite[Section 1.1]{Bru04}, the following normal forms are given for a foliation in the neighbourhood of a reduced singularity, depending on the eigenvalues of $\Xcal$:

$\omega=ydx-\lambda xdy, \lambda\notin \qq^+$ -- this is of the form (A) from \thref{simple}, and is non-resonant;

$\omega=(x(1+\nu y^l))dy-y^{l+1}dx, \nu\in\ccx, l\in\nn$ -- this can be re-arranged to be of the form (B);

$\omega=py(1+g_2(x^py^q))dx+qx(1+g_1(x^py^q))dy, g_i\in\hat{\mm}_1$ -- we can assume $g_2=0$, so this can also be re-arranged to be of the form (B).
\end{remark}

\begin{remark}\thlabel{nonrednf}
We also have normal forms for non-reduced singularities of surface foliations. If the ratio of the eigenvalues of the linear part of $\Xcal$ is $\lambda\in\qq^+$, the foliations can be given by one of the following forms:

$\omega=ydx-\lambda xdy$, which is pre-simple of type (A), but not simple, as there is a resonance;

$\omega=ydx-(rx+ay^r)dy, r\in\nn, a\in\ccx^*$, which is of type (C).

If the linear part of $\Xcal$ is non-zero but has two zero eigenvalues, then there is no SNC divisor with two components tangent to the foliation at the singularity -- if there were, the foliation would be given by the form $yb_1 dx+xb_2 dy$, where the $b_i$ are holomorphic with constant term $\lambda_i$. Then the linear part of $\Xcal$ is $\lambda_1 y\frac{\partial}{\partial y}-\lambda_2 x\frac{\partial}{\partial x}$, which is non-zero only if it has a non-zero eigenvalue, a contradiction.

If there is a smooth curve tangent to the foliation, we choose co-ordinates such that the curve is $E=\vv(y)$, and the foliation is generated by $yb_1 dx+b_2 dy$, where $b_1(x,y)=\lambda_1 +\alpha_1 x+\beta_1 y+\cdots$, $b_2(x,y)=\alpha_2 x+\beta_2 y+\cdots$ (as $b_2$ vanishes at the origin). The conditions on the linear part of $\Xcal$ imply that $\alpha_2 =\lambda_1 =0$, and $\beta_2\neq 0$. Thus we have $\nu(\fcurl,E;0)=1$, and the linear part of $b_2$ depends only on $y$, and so the foliation is not pre-simple.

The foliation can be given by the normal form 
\[
\omega=ydy-(p(x)+yq(x))dx, \ord p\geq 2, \ord q\geq 1.
\]

If the linear part of $\Xcal$ is zero, the singularity is again non-pre-simple.
\end{remark}

\begin{notation}
Let $\icurl^n$ denote the set $\{(i,j)\in\zz^2 \mid 1\leq i<j \leq n\}$.

For $k<n$, let $\icurl^n_k$ denote $\icurl^n \setminus \{(1,k+1),\ldots, (1,n)\}$.
\end{notation}

Let $\fcurl$ be a foliation with a pre-simple singularity at the point $x$, with $\tau(\fcurl,x)=t$, and let $U$ be an open (or formal) neighbourhood of the point on which the foliation is given by one of the normal forms in \thref{simple} (written in the first $t$ co-ordinates). Then, in cases (A) and (B), $\Singr\fcurl\cap U=\bigcup_{(i,j)\in\icurl^t} \{x_i=x_j=0\}$, and in case (C), $\Singr\fcurl\cap U=\bigcup_{(i,j)\in\icurl^t_k} \{x_i=x_j=0\}$.

If $y\in\Singr\fcurl$, we denote the number of the first $t$ co-ordinates which equal zero by $z(y)$.

\begin{lemma}
Let $\fcurl$ be a foliation with a pre-simple singularity of type (A) or (B) at the point $x$, with $\tau(\fcurl,x)=t$, and let $U$ be an open (or formal) neighbourhood of the point on which the foliation is given by one of the normal forms in \thref{simple}. Let $y\in\Singr\fcurl$. Then $\tau(\fcurl,y)=z(y)$.
\end{lemma}

\begin{proof}[Proof: \nopunct]
Let $\omega$ be the $1$-form generating $\fcurl$ near $x$. If $y_i\neq 0$, then $\frac{\partial}{\partial x_i}$ annihilates the $1$-form at $y$. Thus $\tau(\fcurl,y)\leq z(y)$. However, $y$ lies in an SNC divisor $E$ with $z(y)$ components of $E$ through $y$, hence the result holds by \thref{elesst}.
\end{proof}

\begin{definition}
Let $\fcurl$ be a foliation given locally by the $1$-form $\omega=b_1(x_1,\ldots,x_t)dx_1+\cdots +b_t(x_1,\ldots,x_t)dx_t$.
For a fixed point $y=(y_1,\ldots,y_t)$, we define the $1$-form 
\begin{multline}
\omega^i_y=b_1(x_1,\ldots,y_i,\ldots,x_t)dx_1+\cdots+b_{i-1}(x_1,\ldots,y_i,\ldots,x_t)dx_{i-1}+ \nonumber \\
b_{i+1}(x_1,\ldots,y_i,\ldots,x_t)dx_{i+1}+\cdots+b_t(x_1,\ldots,y_i,\ldots,x_t)dx_t;
\end{multline}
denote by $\fcurl^i_y$ the foliation generated by $\omega^i_y$.
Extra indices are added recursively: $\omega^{(i,j)}_y=(\omega^i_y)^j_y$, etcetera, and the same for $\fcurl^{(i,j)}_y$
\end{definition}

\begin{lemma}
Let $\fcurl$ be a foliation given by $\omega$, with a pre-simple singularity at the point $x$, with $\tau(\fcurl,x)=t$, and let $U$ be an open (or formal) neighbourhood of the point on which the foliation is given by one of the normal forms in \thref{simple}. Let $y\in\Singr\fcurl$, and let $I\subset \{1,\ldots,t\}$ be the list of indices of non-zero co-ordinates of $y$. Then $\fcurl$ is generated in a neighbourhood of $y$ by $\omega^I_y$.
\end{lemma}

\begin{proof}[Proof: \nopunct]
If $i\in I$, then $\frac{\partial}{\partial x_i}$ annihilates $\omega$ at $y$. Then in a neighbourhood of $y$, $\fcurl$ is a cylinder over $\fcurl\restn{\{x_i=y_i\mid i\in I\}}$ (see \cite[Lemma 10.8]{Car21}). This foliation is given by the form $\omega^I_y$.
\end{proof}

\begin{theorem}\thlabel{bigthm2f}
Let $\fcurl$ be a codimension-$1$ foliation on $X$, and let $x\in\Singr\fcurl$ be a pre-simple singularity with $\tau(\fcurl,x)=t$. Then $x$ is a simple singularity if and only if there is an open (or formal) neighbourhood $U\ni x$ such that in a local co-ordinate system on $U$, $x=0$,  $\Singr\fcurl\cap U\subset\Singr\vv(x_1\cdots x_t)$, and $J_m(\fcurl,y)$ is isomorphic to $J_m(\vv(x_1\cdots x_t),y)$ for all $y\in \Singr\fcurl\cap U$ and all $m\in \nn_0$.
\end{theorem}

\begin{proof}[Proof: \nopunct]
First, by \thref{simple} we can choose a neighbourhood $U$ of $x$ on which, by a formal co-ordinate change, we can write the generator $\omega$ in one of the above normal forms, and say $x=0$ and that the singular locus is in the right form. Throughout the proof, though we are using formal co-ordinate changes, we can assume the induced isomorphisms on the jet fibres converge by \thref{folfib}.

In calculating the jets, we set $x_i=\sum a_{ij}t^j$, and so $dx_i=(\sum ja_{ij}t^{j-1})dt$, and we equate co-efficients to get the co-efficient of $t^j$.

Now let us suppose $x$ is a simple singularity. We have the following cases:

(A): The co-efficient that vanishes for jets of order $m$ or higher at the origin can be calculated as
\[
\sum_{\{(i_1,\ldots,i_t)\mid i_1+\cdots+i_t=m\}} \left(\sum_{j=1}^t \lambda_j i_j\right) a_{1,i_1}a_{2,i_2}\cdots a_{t,i_t}.
\]
By non-resonance of the $\lambda_i$, none of the co-efficients in the sum vanishes; hence we have $J_m(\fcurl,0)=J_m(\vv(x_1\cdots x_t),0)$.

Now let $y$ be some other singular point, lying in $U$, with $\tau(\fcurl,y)=s$, and with $I$ the list of indices of its non-zero co-ordinates. Then $J_m(\fcurl,y)=J_m(\fcurl^I_y,0)$. As $\omega^I_y=\prod_{1\leq i\leq t, i\notin I}x_i(\sum_{1\leq i\leq t, i\notin I}\lambda_i\frac{dx_i}{x_i})$ is also of the form (A), with the $\lambda_i$ non-resonant, this in turn is isomorphic to $J_m(\vv(x_1\cdots x_s),0)\cong J_m(\vv(x_1\cdots x_t),y)$.

(B): We write $\omega=\omega'+\Psi \omega''$.  For $m<p_1+\cdots +p_k+t-1$, (where the notation is the same as in \thref{simple}), the jets of order $m$ are given solely by the $\omega'$ term. Write $\omega'=x_{k+1}\cdots x_t \eta$, where
 \[
\eta=p_1x_2\cdots x_kdx_1+\cdots+p_kx_1\cdots x_{k-1}dx_k.
\]
 This is of the form (A), and the tuple $(p_1,\ldots,p_k)$ is non-resonant. Letting $\gcurl$ denote the foliation generated by $\eta$, we have $J_m(\gcurl,0)=J_m(\vv(x_1\cdots x_k),0)$, and hence $J_m(\fcurl,0)=J_m(\vv(x_1\cdots x_t),0)$ by \thref{unsat}.

So for low order, 
\begin{multline}
J_m(\fcurl,0)=\{a_{11}=\cdots =a_{1j_1}=a_{21}=\cdots=a_{2j_2}=\cdots  \\ \nonumber=a_{t1}=\cdots=a_{tj_t}=0\mid
j_1+\cdots +j_t=m-t+1\}.
\end{multline}
 Suppose this holds for all orders $m\leq K$. Now every summand in the co-efficient of $t^K$ contributed by the $\Psi\omega''$ term is the product of $p_1+\cdots +p_k+t-1$ terms of the form $a_{ij}$, where the sum of the second indices equals $K$. The equalities from the lower order jets cause these terms to vanish; else there is a summand with $a_i$ terms only of order $>j_i$, contradicting that the sum of the indices is $K$.

Hence by induction, $J_m(\fcurl,0)=J_m(\vv(x_1\cdots x_t),0)$ for all $m$.

Again let $y$ be some other singular point, lying in $U$, with $\tau(\fcurl,y)=s$, and with $I$ the list of indices of its non-zero co-ordinates. Then $J_m(\fcurl,y)=J_m(\fcurl^I_y,0)$. If $I\supset \{1,\ldots, k\}$, then 
\[
\omega^I_y=\prod_{k+1\leq i\leq t, i\notin I}x_i \left(\sum_{k+1\leq i\leq t, i\notin I}\lambda_i\frac{dx_i}{x_i}\right)
\]
is of the form (A), with the $\lambda_i$ non-resonant. Otherwise 
\begin{multline}
\omega^I_y=\prod_{1\leq i\leq t,i\notin I}x_i \left(\sum_{1\leq i\leq k,i\notin I}p_i\frac{dx_i}{x_i}+ \right.\nonumber \\
\left.\Psi\left(\prod_{1\leq i\leq k,i\notin I}x_i^{p_i}\prod_{1\leq i\leq t,i\in I}y_i^{p_i}\right)\sum_{2\leq i\leq t,i\notin I}\lambda_i\frac{dx_i}{x_i}\right)
\end{multline}
 is of the form (B). In either case, we have $J_m(\fcurl,y)\cong J_m(\vv(x_1\cdots x_s),0)\cong J_m(\vv(x_1\cdots x_t),y)$.

Conversely, suppose that $x$ is a pre-simple singularity that is not simple.
We again have the following cases:

(A): Non-simplicity means that the tuple $(\lambda_1,\ldots,\lambda_t)$ is resonant: Let $(r_1,\ldots,r_t)$ be a tuple of non-negative integers (not all zero), such that $\sum_{i=1}^t r_i\lambda_i=0$. (We can choose the $r_i$ so that their sum is minimal among all such tuples).
Let $K=r_1+\cdots +r_t$.

Again, the co-efficient that vanishes for jets of order $m$ or higher at the origin is
\[
\sum_{\{(i_1,\ldots,i_t)\mid i_1+\cdots+i_t=m\}} \left(\sum_{j=1}^t \lambda_j i_j\right) a_{1,i_1}a_{2,i_2}\cdots a_{t,i_t}.
\]
For $m<K$, none of the co-efficients in the sum vanishes; hence we have $J_m(\fcurl,0)=J_m(\vv(x_1\cdots x_t),0)$. For the jets of order $K$, the extra defining equation has a co-efficient of $0$ for the $a_{1,r_1}a_{2,r_2}\cdots a_{t,r_t}$ term, hence, as for the order below, there is a component given by
\[
a_{11}=\cdots =a_{1,r_1-1}=a_{21}=\cdots=a_{2,r_2-1}=\cdots=a_{t1}=\cdots=a_{t,r_t-1}=0,
\]
which is of higher dimension than $J_K(\vv(x_1\cdots x_t),0)$, so there is no isomorphism.

(B): We have $k<t-1$, otherwise the non-resonance condition on the $\lambda_i$ is trivial, and the singularity is automatically simple. Consider the point $y=(y_1,\ldots,y_k,0,\ldots,0)$, where $y_i\neq 0$ for all $i=1,\ldots,k$. This is a singular point with dimensional type $t-k$, and $I=(1,\ldots,k)$. The $1$-form $\omega^I_y$ is of the form (A) and is resonant. So we have
\[
J_m(\fcurl,y)=J_m(\fcurl^I_y,0)\ncong J_m(\vv(x_1\cdots x_{t-k}),0)\cong J_m(\vv(x_1\cdots x_t),y)
\]
for some $m\in\nn$.

(C): The point $y=(0,0,y_3,\ldots,y_t)$, where $y_i\neq 0$ for all $i=3,\ldots,t$, is a singular point of dimensional type $2$, about which the foliation is given by $\omega^I_y$ for $I=(3,\ldots,t)$. Now $\omega^I_y=x_2dx_1-(p_2x_1-\lambda_2\prod_{i=3}^k y_i^{p_i}x_2^{p_2})dx_2$, which for simplicity we can write in the form $ydx-(rx+ay^r)dy$ (see \thref{nonrednf}). We study the behaviour of this form by blowing up at the origin.

On the chart $x=yt$, we get $\omega=y(ydt-((r-1)t+ay^{r-1})dy)$, which is of the same form.
So we take $r$ iterated blow-ups. On one chart we have $x=y^r t$. Then we have $\omega=y^r(ydt-ady)$, which gives a smooth (unsaturated) foliation with fully tangent curve through the origin $\vv(y^{r+1})$, by \thref{unsat} applied to the leaf $\vv(y)$.

On the other charts, we have $x=y^j u, 0\leq j \leq r-1$, which gives $\omega=y^j(ydu-((r-j)u+ay^{r-j})dy)$, followed by $y=ut$. This gives 
\[
\omega=u^{j+1}t^j(((1-r+j)t-au^{r-j-1}t^{r-j+1})du-((r-j)u+au^{r-j}t^{r-j})dt).
\]
 The linear part has ratio of eigenvalues $\frac{1-r+j}{r-j}\leq 0$, so the saturation is reduced. Therefore the unsaturated foliation has fully tangent curve through the origin $\vv(u^{j+2} t^{j+1})$, by \thref{unsat} applied to $\vv(ut)$, which is fully tangent by the argument for simple singularities above.

Blowing back down, we see that the original foliation has fully tangent curve $C=\vv(y^{r+1},xy)$, by \thref{pftang}. Now $J_{r+1}(\vv(y^{r+1}),0)=\vv(b_1)$, and so $J_{r+1}(\fcurl,0)=J_{r+1}(C,0)=J_{r+1}(\vv(y^{r+1},0)\cap J_{r+1}(\vv(xy),0)$ has one fewer irreducible component than $J_{r+1}(\vv(xy),0)$, by \thref{jetex}, so there is no isomorphism. The result follows.
\end{proof}

\begin{theorem}\thlabel{bigthm2b}
Let $\fcurl$ be a codimension-$1$ foliation on $X$, and let $x\in\Singr \fcurl$ be a singularity with $\tau(\fcurl,x)=t$ which is not pre-simple. Then there exists a natural number $m$ such that $J_m(\fcurl,x)$ is not isomorphic to $J_m(\vv(x_1\cdots x_t),0)$.
\end{theorem}

\begin{proof}[Proof: \nopunct]
That $x$ is not a pre-simple singularity means that for any SNC divisor $E\subset X$, either $E$ is not tangent to $\fcurl$, or $x$ is not pre-simple adapted to $E$.

Suppose there is a divisor $E$ with $t$ components through $x$ and tangent to $\fcurl$. Non-pre-simplicity implies that $\nu(\fcurl,E;x)\geq 1$, which means that the unadapted order $\nu(\fcurl,\emptyset;x)\geq t$. Hence the co-efficients of the pullback of $\omega$ under the map $x_i=\sum a_{ij}s^j$ are zero for all $k\leq t$, and so $J_t(\fcurl,x)=\aaf^{tn}$, and there is no isomorphism to $J_t(\vv(x_1\cdots x_t),0)$.

Now suppose our divisor $E$ has $t-1$ components: We can write it as $E=\vv(x_1\cdots x_{t-1})$. There are three cases for non-pre-simplicity:

(i) $\nu(\fcurl,E;x)\geq 2$;

(ii) $\mu(\fcurl,E;x)\geq 2$.

In both these cases we have $\nu(\fcurl,\emptyset;x)\geq t$, and so there is no isomorphism as in the first case.

(iii) $\nu(\fcurl,E;x)=\mu(\fcurl,E;x)=1$, and for all $i=1,\ldots,t-1$, the linear part of $b_i$ is independent of $x_t$. Some of these linear parts are non-zero; denote them by $l_1,\ldots,l_s$ (relabelling co-ordinates if necessary). Then the equation defining the $t$-jets is
\[
a_{1,1}\cdots a_{t-1,1}\left(l_1\left(\sum a_{i,j}t^j\right)+\cdots+l_s\left(\sum a_{i,j}t^j\right)\right)=0.
\]
This is either zero, or else gives at most $t$ components in $t-1$ variables---in either case, there is no isomorphism to $J_t(\vv(x_1\cdots x_t),0)$, which is SNC, with $t$ components in $t$ variables (see \thref{jetex2}).

Now suppose that the only SNC divisors tangent to $\fcurl$ through $x$ have at most $t-2$ components. If we have $J_m(\fcurl,x)\cong J_m(\vv(x_1\cdots x_t),0)$, then by \thref{snctang2}, the pre-image of $\vv(x_1\cdots x_t)$ under the isomorphism is a solution for $\fcurl$, a contradiction.
\end{proof}

Combining \thref{bigthm2f,bigthm2b}, we have:
\begin{corollary}\thlabel{bigthm2}
Let $\fcurl$ be a codimension-$1$ foliation on $X$, and let $x\in\Singr\fcurl$ be a singularity with $\tau(\fcurl,x)=t$. Then $x$ is a simple singularity if and only if there is an open (or formal) neighbourhood $U\ni x$ such that in a local co-ordinate system on $U$, $x=0$,  $\Singr\fcurl\cap U=\Singr\vv(x_1\cdots x_t)$, and $J_m(\fcurl,y)$ is isomorphic to $J_m(\vv(x_1\cdots x_t),y)$ for all $y\in \Singr\fcurl\cap U$ and all $m\in \nn_0$.

Moreover, if $x$ is simple, the union of the separatrices of $\fcurl$ at $x$ is given by $\vv(x_1\cdots x_t)$, which is fully tangent in $U$.
\end{corollary}

\begin{remark}
If $\dimn X=2$, then as the singularities are isolated, we can write the result as follows: The point $x\in\Singr\fcurl$ is a simple singularity if and only if $J_m(\fcurl,x)\cong J_m(\vv(xy),0)$ for all $m$, in which case the union of the separatrices of $\fcurl$ at $x$ is given by $\vv(xy)$, which is fully tangent.
\end{remark}

\begin{remark}
The condition of the jets at the singular locus given in \thref{bigthm2} is not sufficient to recover the normal forms from \thref{simple}. Indeed, if $\fcurl_1$ and $\fcurl_2$ are two foliations on $X$ with the same singular locus $\Sigma$, and such that each singular point of either foliation is a simple singularity, then we have $J_m(\fcurl_1,x)\cong J_m(\fcurl_2,x)$ for each $x\in \Sigma$ and each $m\in\nn$ (see \thref{reducedex}).

However, if we also take into account the jets of the foliation above the smooth locus, we can reconstruct the vector fields that define the foliation; using the eigenvalues of these we can construct the relevant normal form. (See \cite[Section 6]{Can97}.)
\end{remark}

\begin{definition}
Let $\fcurl$ be a singular codimension-$1$ foliation on $X$. A \emph{resolution of singularities} for $\fcurl$ is a sequence of blow-ups $f:X'\rightarrow  X$, where $X'$ is a manifold and $\sat(f^{-1}(\fcurl))$ has only simple singularities.
\end{definition}

\begin{theorem}\thlabel{existres}
Let $\fcurl$ be a singular codimension-$1$ foliation on $X$. Then $\fcurl$ is known to have a resolution of singularities in the following cases:

$\dimn X=2$ \cite{Sei68};

$\dimn X=3$ \cite{Can04};
\end{theorem}

\section{Jets and Separatrices of Non-dicritical Codimension $1$ Foliations}
Throughout this section, we assume that $X$ is a quasi-compact complex manifold, and that $\fcurl$ is a codimension-$1$ foliation on $X$ given by an algebraic $1$-form. Unless stated otherwise, we assume that $\fcurl$ is non-dicritical. In this case, $\fcurl$ has finitely many separatrices by \thref{finsep}.

\begin{definition}
Let $\fcurl$ be a foliation on $X$, and let $Z$ be the union of the separatrices of $\fcurl$. We define
\[
\Cscr(\fcurl)=\{Y\subset X\mid \supp(Y)=Z, \text{$Y$ is strongly tangent to $\fcurl$} \},
\]
where the elements of $\Cscr(\fcurl)$ are formal subschemes of $X$.
\end{definition}

\begin{remark}\thlabel{cfnempt}
$\Cscr(\fcurl)\neq \emptyset$ if and only if $Z$ itself is strongly tangent to $\fcurl$.
\end{remark}

\begin{lemma}\thlabel{cfmax}
If $\Cscr(\fcurl)\neq\emptyset$, then it has a maximal element.
\end{lemma}

\begin{proof}[Proof: \nopunct]
 By \thref{tangform}, all ascending chains of strongly tangent formal subschemes through the singular locus have a strongly tangent upper bound, namely, the direct limit; the result follows by Zorn's lemma.
\end{proof}

\begin{definition}
Such a formal scheme, if it exists, is called the \emph{total separatrix} of $\fcurl$. It is the maximal strongly tangent formal scheme passing through the singular locus.

A non-dicritical foliation with a total separatrix is called \emph{totally separable}.
\end{definition}

\begin{lemma}\thlabel{fttotsep}
If there exists a scheme $C$ supported on the union $Z$ of the separatrices of $\fcurl$ which is fully tangent to $\fcurl$, then $C$ is the total separatrix.
\end{lemma}

\begin{proof}[Proof: \nopunct]
As $C$ is fully tangent and supported on $Z$, we have $J_m(C)=J_m(\fcurl)\restn{Z}$, for all $m\in\nn$. For any scheme $S$ supported on $Z$ and strongly tangent to $\fcurl$, we have for all $m\in\nn$ that $J_m(S)\subset J_m(\fcurl)\restn{Z}$, and therefore $J_m(S)\subset J_m(C)$, for all $m\in\nn$. As $S$ and $C$ have the same support, we have that $S\subset C$ by \thref{jetstruc}, and so $C$ is maximal among all such schemes. Therefore $C$ serves as a total separatrix for $\fcurl$.
\end{proof}

\begin{example}
If a foliation has a first integral, that is, it is generated by a form $\omega=df$ for some function $f$, then it has total separatrix $\vv(f)$. (See \thref{firstint}). In particular, the total separatrix need not be irreducible.
\end{example}

\begin{example}
The foliation of the plane generated by $\omega=ydx-(x+y)dy$ has total separatrix $\vv(xy,y^2)$. So the total separatrix need not be reduced.
\end{example}

\begin{example}
The foliation of the plane given by $\omega=(y-x)dx-x^2dy$ has total separatrix $\vv(xy-\sum_{m=0}^{\infty}m!x^{m+2})$. So the total separatrix need not be an ordinary scheme.
\end{example}

\begin{lemma}\thlabel{embed}
Let $S\subset X$ be a scheme. Let $\pi:X'\rightarrow X$ be a surjective morphism of schemes, with $S'=\pi^{-1}(S)$. Let $Z\subset X'$ be a scheme with the same support as $S'$, such that $S'$ is a strict subscheme of $Z$. Then $S$ is a strict subscheme of $\pi(Z)$.
\end{lemma}

\begin{proof}[Proof: \nopunct]
The scheme $S'$ is the fibre product $S'=S\times_X X'$. Clearly $S\subset \pi(Z)$; if there is equality, then the diagram
\[
\begin{tikzcd}
Z \arrow[r, hook]\arrow[d] & X' \arrow[d, two heads, "\pi"] \\
S \arrow[r, hook]	& X
\end{tikzcd}
\]
commutes, and so by the universal property of fibre products, $Z\cong S'$, a contradiction. The conclusion of the lemma follows. \end{proof}

\begin{lemma}\thlabel{expo}
Consider the function $f:D\rightarrow \ccx, z \mapsto z^{\lambda}$, for some $\lambda\in\ccx$ with $\Rea\lambda\geq 0$, and where $D\subset\ccx$ is an open domain on which the function is well-defined, with $0\in\bar{D}$. (For example, $D$ can be taken as a slitted disc $\Delta\setminus [0,1]$.) 
If $\Rea\lambda >0$, then $\lim_{z\rightarrow 0}f(z)=0$. If $\Rea\lambda =0$, then the limit is undefined.
\end{lemma}

\begin{proof}[Proof: \nopunct]
We write $\lambda=a+bi, a,b\in\rrl, a\geq 0$, and the variable $z$ as $z=re^{i\theta}$. Then $z^{\lambda}=r^a r^{bi} e^{ai\theta} e^{-\theta b}$. The limit as $z$ tends to zero is the limit as $r$ tends to zero: The $\theta$ terms are non-zero and so do not contribute, and the term $r^{bi}=e^{b\logn r i}$ is non-zero but has no limit as $r\rightarrow 0$. Therefore we consider the $r^a$ term. If $a>0$, then $\lim_{r\rightarrow 0}r^a=0$, and so the limit for $z$ is also zero. If $a=0$, then $\lim_{r\rightarrow 0}r^a=1$; the limit for $z$ is thus undefined.
\end{proof}

\begin{proposition}\thlabel{dicrit}
Pre-simple singularities of types (A) and (B) that are not simple are dicritical.
\end{proposition}

\begin{proof}[Proof: \nopunct]
We consider type (A) first. By a suitable choice of formal co-ordinates, the foliation is given locally about the singular point by
\[
\omega=\lambda_1 x_2\cdots x_tdx_1+\lambda_2x_1 x_3\cdots x_tdx_2+\cdots +\lambda_t x_1\cdots x_{t-1}dx_t,
\]
where there is a resonance relation $\sum_{i=1}^t r_i\lambda_i=0, r_i\in\nn_0$.

Claim: Without loss of generality, we may assume that there is at least one $i\in \{1,\ldots,t\}$ with $\Rea \lambda_i>0$, and at least one with $\Rea \lambda_i<0$.

Indeed, if all the real parts have the same sign, then as the resonance relation also holds for the real parts, we must have $r_i=0$ whenever $\Rea \lambda_i\neq 0$. The other $\lambda_i$ have non-zero imaginary part, and as the imaginary parts also satisfy the resonance relation, there must be at least one with $\Ima \lambda_i>0$ and at least one with $\Ima\lambda_i<0$. As we can divide $\omega$ through by $\sqrt{-1}$, the claim is proved.

Now $\omega$ has solutions of the form $x_1^{\lambda_1}\cdots x_t^{\lambda_t}=\mu, \mu\in\ccx$. Take a fixed $\mu\neq 0$. By the claim, we can re-arrange the equation so that on both sides, the $x_i^{\lambda_i}$ terms have $\Rea \lambda_i\geq 0$, with at least one strictly positive term. By \thref{expo}'s results on limits, it follows that the origin lies on this leaf. As $\mu$ was arbitrary, we have that every leaf of the foliation passes through the origin, and so is a separatrix. Therefore the foliation is dicritical about the singularity.

For type (B), we note that the foliation is of type (A) around a component of the singular locus, so we are done.
\end{proof}

\begin{proposition}\thlabel{nonred}
Let $\fcurl$ be a foliation, and let $x$ be a singular point with $\tau(\fcurl,x)=t$, which is adapted to an SNC divisor $E$ with $t$ components but not pre-simple. Then there are non-reduced schemes supported on $E$ that are strongly tangent to $\fcurl$ in a neighbourhood of $x$.
\end{proposition}

\begin{proof}[Proof: \nopunct]
In a neighbourhood of $x$, choose holomorphic co-ordinates so that $x=0$ and the components of the divisor are given by $x_i=0$. $\fcurl$ is given by the $1$-form
\[
\omega=x_2\cdots x_tb_1dx_1+x_1x_3\cdots x_t b_2dx_2+\cdots+ x_1\cdots x_{t-1}b_t dx_t,
\]
where $b_i=b_i(x_1,\ldots,x_t)$ are holomorphic functions of order at least $1$. (See \thref{elesst}.)
We now blow up at the origin: For $j=1,\ldots,t$, the map is given in the corresponding chart by $x_i=x_j v_i, i\neq j$.

In the $j$th chart, the pre-image foliation is given by the form
\[
x_j^{t-1}\prod_{i=1,i\neq j}^t v_i\left(\sum_{i\neq j}x_j b_i \frac{dv_i}{v_i}+ (b_1+\cdots +b_t)dx_j\right),
\]
where $b_i=b_i(x_jv_1,\ldots,x_jv_{j-1},x_j,x_jv_{j+1},\ldots,x_jv_t)$. As the order of each $b_i$ is at least $1$, in the new co-ordinates we can factor out a further copy of $x_j$. 
It is clear that $\vv(x_j)$ is a solution for the transformed foliation, as is $\vv(v_i)$ for all $i\neq j$---they are therefore strongly tangent by \thref{redstt}. These hyperplanes form the components of an SNC divisor, so by \thref{snctang,unsat}, we therefore have that $\vv(v_1\cdots v_{j-1}x_j^{t+1} v_{j+1}\cdots v_t)$ is strongly tangent to the transformed foliation. Blowing back down each chart by replacing $v_i$ with $\frac{x_i}{x_j}$, we see by \thref{pftang} that $\vv(x_1^2 x_2\cdots x_t,\ldots, x_1\cdots x_{t-1}x_t^2)$ is strongly tangent to $\fcurl$.
\end{proof}

\begin{lemma}\thlabel{eeqt}
Suppose a foliation $\fcurl$ has total separatrix $E$, an NC divisor. Then for each point $x\in \Singr\fcurl$, $e(E,x)=\tau(\fcurl,x)$ (where $e(E,x)$ is the number of local components of $E$ through $x$).
\end{lemma}

\begin{proof}[Proof: \nopunct]
Suppose not. Let $\tau(\fcurl,x)=t$, and choose holomorphic co-ordinates about $x$ so that $\fcurl$ is given by $\omega=b_1dx_1+\cdots +b_tdx_t$, where $b_i$ are holomorphic, and that $x_1=0,\ldots,x_k=0$, are the components of $E$ through $x$. By \thref{elesst}, we have $e(E,x)\leq \tau(\fcurl,x)$, and so here we have $k< t$.

For $1\leq i\leq k$, we can write $b_i=x_1\cdots \hat{x_i}\cdots x_k b'_i$, where the hat denotes omission; for $i>k$ we can write $b_i=x_1\cdots x_k b_i'$. 
Thus the vector fields
\begin{multline}
x_2 b_1'\frac{\partial}{\partial x_2}-x_1 b_2'\frac{\partial}{\partial x_1},\ldots, x_k b_1'\frac{\partial}{\partial x_k}-x_1 b_k'\frac{\partial}{\partial x_1}, \\ \nonumber
 b_1'\frac{\partial}{\partial x_{k+1}}-x_1b_{k+1}'\frac{\partial}{\partial x_1},\ldots, b_1'\frac{\partial}{\partial x_t}-x_1b_t'\frac{\partial}{\partial x_1}, \frac{\partial}{\partial x_{t+1}},\ldots, \frac{\partial}{\partial dx_n}
\end{multline}
annihilate $\omega$ at $x$. As the dimensional type is $t$, we conclude that $b_1'(x)=0$. Using a different collection of vector fields, we also have that $b_i'(x)=0$, for all $i\in\{2,\ldots,k\}$, and hence $\ord b_i\geq 1, 1\leq i\leq k$.

We now blow up at $x$, with the charts given by $x_i=x_j v_i,i=1,\ldots,t$. 

If $j\leq k$, then in the $j$th chart, the pre-image foliation is given by the form
\begin{multline}
x_j^{k-1}\left(\sum_{i=1,i\neq j}^kv_1\cdots v_{j-1}x_j v_{j+1}\cdots v_kb'_i\left(\frac{dv_i}{v_i}+\frac{dx_j}{x_j}\right)+ \right. \\ \nonumber
\left. v_1\cdots v_{j-1} v_{j+1}\cdots v_kb'_jdx_j+ 
\sum_{i=k+1}^t v_1\cdots v_{j-1}x_j v_{j+1}\cdots v_k b'_i(x_jdv_i+v_idx_j)\right).
\end{multline}
As the order of each $b_i, 1\leq i\leq k$, is at least $1$, in the new co-ordinates we can factor out a further copy of $x_j$.
It is clear that $\vv(x_j)$ is a solution of the transformed foliation, and hence strongly tangent by \thref{redstt}, as is $\vv(v_i)$ for all $i\neq j, i\leq k$. By \thref{snctang,unsat}, we therefore have that $\vv(v_1\cdots v_{j-1}x_j^{k+1} v_{j+1}\cdots v_k)$ is strongly tangent to the transformed foliation.

If $j>k$, then in the $j$th chart, the pre-image foliation is given by the form
\begin{multline}
x_j^k\left(\sum_{i=1}^k v_1\cdots v_kb'_i\left(\frac{dv_i}{v_i}+\frac{dx_j}{x_j}\right)+ \right. \nonumber \\
\left.\sum_{i=k+1,i\neq j}^t v_1\cdots v_kb'_i(x_jdv_i+v_idx_j)+v_1\cdots v_kb'_jdx_j\right).
\end{multline}

Again we see that $\vv(v_1),\ldots,\vv(v_k)$ are strongly tangent to the transformed foliation; so is $\vv(x_j)$, as we can factor out a copy of $x_j$ from each of the $b_i, 1\leq i\leq k$. Thus by \thref{snctang,unsat} $\vv(v_1\cdots v_k x_j^{k+1})$ is strongly tangent to the transformed foliation. Blowing back down each chart, we see by \thref{pftang} that 
\[
\vv(x_1^2 x_2\cdots x_k,\ldots, x_1\cdots x_{k-1}x_k^2,x_1\cdots x_k x_{k+1},\ldots, x_1\cdots x_k x_t)
\]
is strongly tangent to $\fcurl$, thus contradicting maximality of $E$.
\end{proof}

\begin{proposition}\thlabel{tssnc}
A foliation $\fcurl$ has all singularities being simple if and only if it is totally separable, and its total separatrix is a normal crossings divisor.
\end{proposition}

\begin{proof}[Proof: \nopunct]
Suppose all singularities are simple. Then by \thref{bigthm2}, at each singular point $x$, there exists an SNC divisor $E_x$, whose components are the separatrices, and which is fully tangent on a neighbourhood of $x$. The $E_x$ can be glued together to give an NC divisor $C$, which is fully tangent to $\fcurl$. Thus $C$ is the total separatrix.

\sloppy Conversely, if the total separatrix $C$ exists and is NC, then by \thref{eeqt}, the number of local components of $C$ through a singular point $x$ is $\tau(\fcurl,x)$. \thref{nonred} excludes the case that $x$ is non-pre-simple; \thref{dicrit} excludes the case that $x$ is pre-simple of type (A) or (B) but not simple. 
Suppose $x$ is pre-simple of type (C). Then there exists a point $y$ in a neighbourhood of $x$ with $\tau(\fcurl,y)=2$ and only one component of $C$ through $y$, (see the proof of \thref{bigthm2f}). This is a contradiction to \thref{eeqt}, so we have the result.
\end{proof}

\begin{theorem}\thlabel{bigthm3}
Let $\fcurl$ be a non-dicritical codimension-$1$ foliation on $X$. Then $\fcurl$ admits a resolution to simple (or reduced, if $\dimn X=2$) singularities if and only if $\fcurl$ is totally separable, and its total separatrix $C$ admits a resolution to a normal crossings support divisor. In this case, the resolutions can be assumed to be the same map.
\end{theorem}

\begin{proof}[Proof: \nopunct]
Suppose $C$ exists, and admits a log resolution $\pi:X'\rightarrow X$. Let $\hat{\fcurl}=\sat(\pi^{-1}(\fcurl))$. Suppose $\hat{\fcurl}$ is given locally by the form $\omega$, so that $\pi^{-1}(\fcurl)$ is given by $f\omega$, for some holomorphic function $f$. If $\pi^{-1}(C)$ is given by the function $g$, we let $\hat{C}$ be the complex space defined by $\frac{g}{f}$. (This is indeed holomorphic, as $f$ defines the non-reduced structure from the exceptional divisors. We can recover $\pi^{-1}(C)$ from $\hat{C}$ by adding in the exceptional divisors.)

Now $\hat{C}$ is SNC, and is strongly tangent to $\hat{\fcurl}$ (by \thref{snctang}, as each component is a separatrix and hence strongly tangent).
Moreover, $\hat{C}$ is the total separatrix. Indeed, if there exists a larger strongly tangent scheme supported on $\hat{C}$, then by \thref{unsat}, adding in the exceptional divisors of $\pi$ yields a scheme strongly tangent to $\pi^{-1}(\fcurl)$ and strictly larger than $\pi^{-1}(C)$; by \thref{pftang} blowing down yields a scheme strongly tangent to $\fcurl$ and strictly larger than $C$ by \thref{embed}, a contradiction.

Therefore by \thref{tssnc}, $\hat{\fcurl}$ has all singularities simple, so $\pi$ is a resolution of singularities for $\fcurl$.

Conversely, suppose there is a resolution $\pi:X'\rightarrow X$ such that $\hat{\fcurl}=\sat(\pi^{-1}(\fcurl))$ has only simple singularities. By \thref{tssnc}, $\hat{\fcurl}$ is totally separable, and its total separatrix $E$ is NC and fully tangent.  By \thref{unsat}, adding in the exceptional divisors of $\pi$ yields a scheme fully tangent to $\pi^{-1}(\fcurl)$; blowing down yields a scheme $C$, which is fully tangent to $\fcurl$ by \thref{pftang}, and hence the total separatrix for $\fcurl$ by \thref{fttotsep}, and which is resolved by $\pi$.
\end{proof}

\begin{corollary}\thlabel{tsft}
If $\fcurl$ admits a resolution, then its total separatrix is fully tangent.
\end{corollary}

\begin{definition}\thlabel{trutran}
Let $\fcurl$ be a codimension-$1$ foliation on $X$. A hypersurface $V\subset X$ is said to be \emph{truly transversal} to $\fcurl$ if it is smooth, reduced and irreducible; if it is not tangent to $\fcurl$; and if the restriction foliation $\fcurl\restn{V}$ is saturated.
\end{definition}

The condition that the restriction foliation is saturated means that $V$ does not contain any codimension-$2$ component of the singular locus, and its intersection with any leaf has codimension at least $2$. By taking hypersurfaces of sufficiently high degree, a truly transversal hypersurface can be found through every point $x\in X$.

\begin{proposition}\thlabel{antotsep}
Let $\fcurl$ be a non-dicritical and totally separable foliation on $X$, and suppose that each of its separatrices is analytic. Then the total separatrix is analytic.
\end{proposition}

\begin{proof}[Proof: \nopunct]
First suppose that $\dimn X=3$. Then $\fcurl$ admits a resolution $\pi:X'\rightarrow X$, and the total separatrix $\hat{C}$ of $\sat(\pi^{-1}(\fcurl))$ is the union of the transforms of the separatrices of $\fcurl$ and the exceptional divisors of $\pi$. By assumption, this is analytic. The total separatrix of $\fcurl$ is obtained by thickening certain components of $\hat{C}$ (the exceptional divisors) and then blowing back down. It follows that this too is analytic.

Now suppose that the result holds whenever $\dimn X=k\geq 3$. Let $\fcurl$ be a foliation on $X$ with $\dimn X=k+1$ with each of its separatrices being analytic but with the total separatrix $C$ being a strict formal scheme. Let $V$ be a truly transversal hypersurface. Then the separatrices of $\fcurl\restn{V}$ are either separatrices or general leaves of $\fcurl$ intersected with $V$ -- in particular, they are all analytic, so by induction the total separatix $C_V$ of $\fcurl\restn{V}$ is analytic. However, if $V$ is sufficiently general, $C\cap V$ is a strict formal scheme which is fully tangent, and therefore a component of $C_V$, a contradiction.
\end{proof}

\begin{theorem}\thlabel{anares}
Let $\fcurl$ be a non-dicritical foliation on $X$, and suppose that each of its separatrices is analytic. Then $\fcurl$ admits a resolution to simple singularities.
\end{theorem}

\begin{proof}[Proof: \nopunct]
Let $Z$ be the union of the separatrices of $\fcurl$, and let $\pi:X'\rightarrow X$ be a log resolution of $Z$. Then the separatrices of $\hat{\fcurl}=\sat(\pi^{-1}(\fcurl))$ are the transforms of the separatrices of $\fcurl$ along with the exceptional divisors of $\pi$; their union is an (analytic) SNC divisor. As each component is tangent, their union is strongly tangent to $\hat{\fcurl}$ by \thref{snctang}. Therefore $\hat{\fcurl}$ is totally separable (see \thref{cfmax}). By \thref{antotsep} the total separatrix is analytic, and so admits a resolution.

Therefore, by \thref{bigthm3}, $\hat{\fcurl}$ admits a resolution to simple singularities. By extension $\pi^{-1}(\fcurl)$ admits a resolution, and hence $\fcurl$ itself admits a resolution to simple singularities.
\end{proof}

In the case where $\fcurl$ has a formal separatrix, we begin by considering germs.

\begin{definition}
Let $\fcurl$ be a codimension-$1$ foliation on $X$ given by $\omega$, and suppose in some co-ordinate chart we have $0\in\Singr\fcurl$. The \emph{germ} of $\fcurl$ at $0$, denoted $\fcurl_0$, is given by writing $\omega$ in the germs of the co-ordinate functions at $0$. It is defined on the formal completion $\hat{X}_0$.
\end{definition}

By choosing appropriate co-ordinates, we can define the germ of $\fcurl$ at any singular point.

\begin{lemma}\thlabel{ptfres}
Let $Y$ be a formal scheme supported at a point $P\in X$. Then there exists a desingularisation $f:X'\rightarrow \hat{X}_P$, with $X'$ smooth, such that $f^{-1}(Y)$ has simple normal crossing support.
\end{lemma}

\begin{proof}[Proof: \nopunct]
As the underlying topological space of $Y$ is a point, by \cite[Theorem 6.29]{Car21}, $Y$ is countably indexed, is a classical formal scheme, and is a closed formal subscheme of $\hat{X}_P$. As $\hat{X}_P$ is quasi-compact and quasi-excellent, the desingularisation exists by \cite[Theorem 1.1.9, Theorem 1.1.13]{Tem18}.
\end{proof}

\begin{proposition}\thlabel{ndtotsep}
Every non-dicritical foliation $\fcurl$ on $X$ is totally separable. Moreover, the total separatrix is fully tangent and is unique.
\end{proposition}

\begin{proof}[Proof: \nopunct]
Let $Z$ be the union of the separatrices of $\fcurl$. If $Z$ is analytic, then by \thref{anares} and \thref{tsft}, $\fcurl$ is totally separable, and its total separatrix is fully tangent.

If $Z$ is a strict formal scheme, we consider the germ of $Z$ at some point $P$, which is given by a formal power series. By \thref{ptfres}, the germ $Z_P$ admits a desingularisation $f:X'\rightarrow \hat{X}_P$, with $f^{-1}(Z_P)$ having SNC support. Using the arguments from the proof of \thref{anares}, $f^{-1}(Z_P)$ is strongly tangent to $f^{-1}(\fcurl_P)$, and hence by \thref{pftang} $Z_P$ is strongly tangent to $\fcurl_P$. By taking germs of $Z$ around different points, we see that $Z$ is strongly tangent to $\fcurl$, and hence $\fcurl$ is totally separable.

Now let $C$ be the total separatrix of $\fcurl$, and let $P\in\Singr\fcurl$. Then $C_P=C\cap \hat{X}_P$ is the total separatrix for the germ $\fcurl_P$. By \thref{ptfres} and \thref{bigthm3}, $\fcurl_P$ can be desingularised (as a sheaf on the formal completion), and so by \thref{tsft}, $C\cap \hat{X}_P$ is fully tangent to $\fcurl_P$. Then for all $m\in\nn$
\[
J_m(C,P)=J_m(C\cap \hat{X}_P,P)=J_m(\fcurl_P,P)=J_m(\fcurl,P).
\]
As this holds for any singular point $P$, and $C$ is clearly fully tangent over the smooth locus, it follows that $C$ is fully tangent to $\fcurl$.

Suppose $C,C'$ are two candidates for the total separatrix. They are both fully tangent, and both have the same support. Then for  any $x\in C, m\in\nn$, $J_m(C,x)=J_m(\fcurl,x)=J_m(C',x)$. It follows from \thref{jetstruc} that $C=C'$.
\end{proof}

By these results, we have:
\begin{theorem}\thlabel{bigthm4}
Let $\fcurl$ be a non-dicritical foliation on $X$, and let $\fcurl_P$ be the germ of the foliation about some singular point. Then $\fcurl_P$ admits a resolution to simple singularities. In the case where all separatrices of $\fcurl$ are analytic, this resolution applies globally.
\end{theorem}

In the case where all separatrices are analytic, the global resolution from \thref{anares} can be recovered from the germ resolution by adding back in the separatrices. Similarly, in the case where the union of separatrices, and hence the total separatrix, is analytic outside of a single point, taking a resolution on the level of germs we can attain a global resolution by adding in the convergent separatrices. If the locus where the union of separatrices is a strict formal schemes is positive dimensional, taking resolutions at multiple germs may be required.

\section{Results on Dicritical Foliations}
\subsection{The Hull of Jets}
We now present an alternate characterisation of dicriticality.

\begin{lemma}\thlabel{thickpt}
Let $Y$ be a formal scheme supported on a point $P\in X$, and let $r\in\nn$. Suppose that $J_k(Y,P)=\aaf^{kn}$ for all $k\leq r$. Then $P^{r+1}\subset Y$.
\end{lemma}

\begin{proof}[Proof: \nopunct]
Suppose $Y$ is an ordinary scheme, and consider the intersection of all schemes satisfying the condition on the jets. This is contained in $Y$. $P^{r+1}$ also satisfies the condition; by looking at the generators we see that any proper subscheme of $P^{r+1}$ does not. Therefore $P^{r+1}$ is equal to the intersection, and is contained in $Y$.

Now suppose $Y=\varinjlim Y_{\lambda}$. By the proof of \cite[Theorem 6.29]{Car21}, one of the $Y_{\lambda}$ satisfies the condition on the jets; hence $P^{r+1}\subset Y_{\lambda}\subset Y$.
\end{proof}

\begin{proposition}\thlabel{thickhs}
Let $Y$ be a formal scheme supported on a smooth hypersurface $H\subset X$, and let $r\in\nn$. Suppose that $J_k(Y,x)=\aaf^{kn}$ for all $k\leq r$ and all $x\in H$. Then $H^{r+1}\subset Y$.
\end{proposition}

\begin{proof}[Proof: \nopunct]
Let $x\in H$. Then for $k\leq r$, $J_k(Y\cap \hat{X}_x,x)=J_k(Y,x)\cap J_k(\hat{X}_x,x)=\aaf^{kn}$. Then $Y\cap \hat{X}_x$ satisfies the assumptions of \thref{thickpt}, and so contains $x^{r+1}$. This holds for every point $x\in H$, so it follows that $Y$ contains $H^{r+1}$.
\end{proof}

\begin{lemma}\thlabel{ftsing}
Let $\fcurl$ be a foliation, and suppose there is a fully tangent formal scheme $Y$ supported on the singular locus. Then for any sequence of blow-ups $\pi:X'\rightarrow X$, $\pi^{-1}(\fcurl)$ has a fully tangent formal scheme supported on its singular locus.
\end{lemma}

\begin{proof}[Proof: \nopunct]
By \thref{tanginv}, $\pi^{-1}(Y)$ is fully tangent to $\pi^{-1}(\fcurl)$, and is supported on its singular locus.
\end{proof}

\begin{proposition}\thlabel{dicritchar}
Let $\fcurl$ be a foliation on $X$. Then $\fcurl$ is non-dicritical if and only if there is a (unique) fully tangent formal scheme $Y$ supported on the singular locus $\Sigma=\Singr\fcurl$.
\end{proposition}

\begin{proof}[Proof: \nopunct]
Suppose $\fcurl$ is non-dicritical. Then by \thref{ndtotsep}, $\fcurl$ admits a total separatrix $C$, which is fully tangent. Then $C\cap \hat{X}_{\Sigma}$ is fully tangent and supported on the singular locus.

If $Y$ is another fully tangent formal scheme supported on $\Sigma$, then for each $m\in\nn$ and $x\in Y$, $J_m(Y,x)=J_m(\fcurl,x)=J_m(C\cap \hat{X}_{\Sigma},x)$. As they have the same support, it follows that $Y=C\cap \hat{X}_{\Sigma}$.

Now suppose $\fcurl$ is dicritical. Then there is a sequence of blow-ups $\pi:X'\rightarrow X$ such that one of the exceptional divisors of $\pi$ is transversal to the leaves of $\sat(\gcurl)$, where $\gcurl=\pi^{-1}(\fcurl)$. On the smooth locus of $\sat(\gcurl)$ we can write a generator as $dx_2$ in some co-ordinate chart; as the exceptional divisor is transversal to this, we can choose holomorphic co-ordinates on some open set $U\subset X'$, contained in the smooth locus of $\sat(\gcurl)$, such that the underlying reduced scheme of the divisor is $\vv(x_1)$. Hence $\gcurl$ is given by the $1$-form $x_1^{r}dx_2$, for some $r\in\nn$.

Suppose there is a fully tangent formal scheme $Y$ supported on $\Singr\fcurl$. Then by \thref{ftsing}, there is a formal scheme $Y'$ fully tangent to $\gcurl\restn{U}$ and supported on $H=\vv(x_1)$. Now for $k\leq r$, $J_k(Y',x)=\aaf^{kn}$, for all $x\in H$. So by \thref{thickhs}, $H^{r+1}=\vv(x_1^{r+1})\subset Y'$. However, $H^{r+1}$ is not strongly tangent to $\gcurl$: The $(2r+1)$-jet given by $x_1=t^2, x_2=t, x_3=\cdots =x_n=0$ is a jet in $J_{2r+1}(\vv(x^{r+1}),0)$ but not in $J_{2r+1}(\gcurl,0)$. This is a contradiction, so the formal scheme $Y$ does not exist.
\end{proof}

A candidate for the formal scheme $Y$ in \thref{dicritchar} can be constructed as follows:

Let $\fcurl$ be a foliation on $X$ with singular locus $\Sigma$. We define $S_m(\fcurl)$ to be the smallest formal subscheme supported on $\Sigma$ such that $J_k(\fcurl,x)\subset J_k(S_m(\fcurl),x)$ for all $k\leq m$ and all $x\in\Sigma$. We define the \emph{hull of jets} to be the formal scheme $\Sscr(\fcurl)=\varinjlim S_m(\fcurl)$. As $S_m(\fcurl)\subset \Sigma^{m+1}$, we have $\Sscr(\fcurl)\subset \hat{X}_{\Sigma}$.

\begin{proposition}\thlabel{sfmin}
$\Sscr(\fcurl)$ is the smallest formal scheme supported on $\Sigma$ such that $J_k(\fcurl,x)\subset J_k(\Sscr(\fcurl),x)$ for all $k\in\nn$ and all $x\in\Sigma$.
\end{proposition}

\begin{proof}[Proof: \nopunct]
Let $Y$ be another such formal scheme. For each $m\in\nn$, we have $J_k(\fcurl,x)\subset J_k(Y,x)$ for all $k\leq m$ and all $x\in\Sigma$. Hence $S_m(\fcurl)\subset Y$, for all $m\in\nn$, and so $\Sscr(\fcurl)\subset Y$.
\end{proof}

We can then reformulate \thref{dicritchar} as follows:
\begin{proposition}\thlabel{hulldicrit}
Let $\fcurl$ be a foliation on $X$ with singular locus $\Sigma$. Then $\fcurl$ is dicritical if and only if $\Sscr(\fcurl)$ is not strongly tangent to $\fcurl$.
\end{proposition}

\begin{proof}[Proof: \nopunct]
If $\fcurl$ is dicritical, then by \thref{dicritchar} there is no formal scheme supported on $\Sigma$ which is fully tangent. In particular, $\Sscr(\fcurl)$ is not fully tangent, and hence not strongly tangent.

Conversely, if $\Sscr(\fcurl)$ is not strongly tangent, then by \thref{sfmin}, any formal scheme $Y$ supported on $\Sigma$ with $J_m(\fcurl,x)\subset J_m(Y,x)$ for all $m\in\nn$ and all $x\in\Sigma$ is not strongly tangent, and hence none of them is fully tangent. By \thref{dicritchar}, $\fcurl$ is dicritical.
\end{proof}

If $\Sscr(\fcurl)$ is strongly tangent, then $\fcurl$ is non-dicritical. We have $\Sscr(\fcurl)=C\cap \hat{X}_{\Sigma}$, where $C$ is the total separatrix.
The codimension-$1$ components of $\Sscr(\fcurl)$ are the germs of the separatrices, so if we start with $\Sscr(\fcurl)$ we can reconstruct the total separatrix by taking the union with those separatrices which are ordinary schemes.

\subsection{Jets of Dicritical Foliations}
\begin{proposition}
Let $\fcurl$ be a dicritical foliation on $X$ which has a resolution to simple singularities. Then there is a family of strongly tangent subschemes $(C_{\alpha})_{\alpha\in A}$ of $X$ such that $J_m(\fcurl,x)=\bigcup_{\alpha\in A}J_m(C_{\alpha},x)$, for all $m\in\nn$ and all $x\in\Singr\fcurl$.
\end{proposition}

\begin{proof}[Proof: \nopunct]
We can take $\pi$ to be a partial resolution of $\fcurl$, such that $\gcurl=\sat(\pi^{-1}(\fcurl))$ is non-dicritical. Then $\gcurl$ has a resolution to simple singularities (by completing the resolution of $\fcurl$), and so by \thref{bigthm3} is totally separable, with fully tangent total separatrix. Let $D$ be the exceptional divisor of $\pi$. As $\fcurl$ is dicritical, at least one of the components of $D$ is transversal to the leaves of $\gcurl$.

For each $x\in D$, let $L_x$ be either the unique leaf of $\gcurl$ through $x$, or the total separatrix of $\gcurl$, as appropriate. Then for each $m\in\nn$, $J_m(\gcurl,x)=J_m(L_x,x)$. So by \thref{unsat}, $J_m(\pi^{-1}(\fcurl),x)=J_m(L_x\cup D,x)$. Using \thref{pftang} we blow back down to get $J_m(\fcurl,x)=\bigcup J_m(\pi(L_x\cup D),x)$, which yields the result.
\end{proof}

\begin{example}
Let $X=\aaf^2$, and let $\fcurl$ be given by $\omega=y^2dx-x^2dy$. This has a single dicritical singularity at the origin. The leaves of the foliation are $\{x=0\}$, $\{y=0\}$, $\{y=x\}$, and $\{xy+kx-ky=0\},k\in\ccx^*$, all of which are separatrices.

We blow up at the origin. In the first chart we set $y=xv$, and get a new foliation given by $x^2((v^2-v)dx-xdv)$. This has two singularities: $(0,0)$, which is reduced, and $(0,1)$, which is again dicritical.

We blow up this second singularity. In one chart we set $v-1=xt$, which leaves us with $x^4(t^2dx-dt)$. The saturation of this foliation is smooth, and has leaves $\vv(t)$ and $\vv(xt+kt+1),k\in \ccx$. In the other chart we set $x=(v-1)s$, which gives $(v-1)^4 s^2(vds+sdv)$, the saturation of which has a single simple singularity.

So the jets of the unsaturated foliation along the exceptional divisor are of the form $J_m(\vv(x^4(xt+kt+1)),(0,-1/k))$ and $J_m(\vv(x^4t),(0,0))$ in one chart, and $J_m(\vv(v(v-1)^4 s^3),(0,1))$ in the other. Blowing down, the jets are of the form $J_m(\vv(x^3v(v-1)))$ and $J_m(\vv(x^3(xv+kv-k)))$.

In the second chart of the initial blow-up, we set $x=yb$ to get the form $y^2((b-b^2)dy+ydb)$. This also has two singularities: $(0,0)$, which is reduced, and corresponds to the reduced singularity in the first chart, and $(1,0)$, which is dicritical. Applying the same method as before, we see that the jets are of the form $J_m(\vv(b(b-1)y^3))$ and $J_m(\vv(y^3(yb+kb-k)))$.

Blowing down again, we have
\begin{multline}
J_m(\fcurl,0)=J_m(\vv(x^2y-xy^2),0)\cup \nonumber \\
\bigcup_{k\in\ccx^*}J_m(\vv(x^2(xy+kx-ky),y^2(xy+kx-ky)),0).
\end{multline}
\end{example}

\bibliographystyle{myplain}

\bibliography{citations}

\end{document}